\documentclass[12pt]{amsart}
\usepackage{latexsym, url}
\usepackage{amsthm}
\usepackage{amsmath}
\usepackage{amsfonts}
\usepackage{amssymb}
\usepackage[dvips]{graphicx}
\usepackage{xypic}
\usepackage{xcolor}
\usepackage{tikz}
\usetikzlibrary{arrows}
\usetikzlibrary{calc}
\usepackage{hyperref}
\input xy
\xyoption{all}

\newtheorem{theo}{Theorem}[section]
\newtheorem*{theo1}{Theorem}
\newtheorem{lemma}[theo]{Lemma}
\newtheorem{assume}[theo]{Assumption}
\newtheorem{propo}[theo]{Proposition}
\newtheorem{coro}[theo]{Corollary}
\newtheorem*{theo:main}{Theorem~\ref{inj1}}

\theoremstyle{definition}
\newtheorem{defi}[theo]{Definition}
\newtheorem{nota}[theo]{Notation}
\newtheorem{rem}[theo]{Remark}

\newtheorem{exam}[theo]{Example}
\newtheorem{exams}[theo]{Examples}

\newcommand\Mod{\operatorname{\bf Mod}}

\newcommand\op{\operatorname{op}}
\newcommand\id{\operatorname{id}}

\newcommand\Set{\operatorname{\bf Set}}
\newcommand\Met{\operatorname{\bf Met}}
\newcommand\Cat{\operatorname{\bf Cat}}

\newcommand\Str{\operatorname{\bf Str}}
\newcommand\CMet{\operatorname{\bf CMet}}
\newcommand\Ab{\operatorname{\bf Ab}}
\newcommand\Ban{\operatorname{\bf Ban}}
\newcommand\Pos{\operatorname{\bf Pos}}

\newcommand\CPO{\operatorname{\bf CPO}}

\newcommand\colim{\operatorname{colim}}

\newcommand\eps{\varepsilon}

\newcommand\ca{\mathcal {A}}

\newcommand\cc{\mathcal {C}}

\newcommand\cg{\mathcal {G}}
\newcommand\ch{\mathcal {H}}

\newcommand\ce{\mathcal {E}}

\newcommand\ck{\mathcal {K}}

\newcommand\cm{\mathcal {M}}

\newcommand\cp{\mathcal {P}}

\newcommand\cv{\mathcal {V}}

\newcommand{\RR}{{\mathbb R}}
\newcommand{\LL}{{\mathbb L}}

\newcommand{\TT}{{\mathbb T}}
\newcommand{\FF}{{\mathbb F}}

\newcommand{\tx}{\textnormal}
\newcommand{\bo}{\mathbf}
\date{January 3, 2025}
 
\begin{document}
\title[Enriched concepts of regular logic]
{Enriched concepts of regular logic}
\author[J. Rosick\'{y} and G. Tendas]
{J. Rosick\'{y} and G. Tendas}
\address{
\newline J. Rosick\'{y}\newline
Department of Mathematics and Statistics\newline
Masaryk University, Faculty of Sciences\newline
Kotl\'{a}\v{r}sk\'{a} 2, 611 37 Brno, Czech Republic\newline
\textnormal{rosicky@math.muni.cz}\newline
\newline G. Tendas\newline
Department of Mathematics\newline
University of Manchester, Faculty of Science and Engineering\newline
Alan Turing Building, M13 9PL Manchester, UK\newline
\textnormal{giacomo.tendas@manchester.ac.uk}\vspace{5pt}\newline
\textit{Secondary address:}\newline
Department of Mathematics and Statistics\newline
Masaryk University, Faculty of Sciences\newline
Kotl\'{a}\v{r}sk\'{a} 2, 611 37 Brno, Czech Republic
}
 
\begin{abstract}
	Building on our previous work on enriched universal algebra, we define a notion of enriched language consisting of function and relation symbols whose arities are objects of the base of enrichment $\cv$. In this context, we construct atomic formulas and define the {\em regular fragment} of our enriched logic by taking conjunctions and existential quantification of those. We then characterize $\cv$-categories of models of regular theories as enriched injectivity classes in the $\cv$-category of structures. These notions rely on the choice of an orthogonal factorization system $(\ce,\cm)$ on $\cv$ which will be used, in particular, to interpret relation symbols and existential quantification.
\end{abstract} 

\maketitle

\setcounter{tocdepth}{1}
\tableofcontents
\section{Introduction}
After considering a notion of enriched equational logic in~\cite{RTe}, we now proceed to introduce the regular fragment of enriched logic. More precisely, we go beyond~\cite{RTe} by allowing our enriched languages to involve function and relation symbols whose arities are objects of the base of enrichment. Then, we introduce corresponding enriched notions of atomic formulas, as well as conjunctions and existential quantifications of such.
 
To achieve this, we equip the base $\cv$ of enrichment with an enriched factorization system $(\ce,\cm)$.
We use the class $\cm$ to interpret the relation symbols $R$, of a given language $\LL$, as $\cm$-subobjects 
\begin{center}
	\begin{tikzpicture}[baseline=(current  bounding  box.south), scale=2]

		\node (d0) at (0,0) {$R_A$};
		\node (c1) at (1,0) {$ A^X,$};
		
		\path[font=\scriptsize]
		
		(d0) edge [>->] node [above] {$\cm$} (c1);
	\end{tikzpicture}	
\end{center}
where $X$ is the arity of $R$, and $A$ is an object of $\cv$ that we are endowing with an $\LL$-structure.  

Building on~\cite{RTe}, we define the $\cv$-category $\Str(\LL)$ of $\LL$-structures, prove that it is locally $\lambda$-presentable as a $\cv$-category (when $\LL$ is $\lambda$-ary), and characterize its $\lambda$-presentable objects. 

Given a language $\LL$, we form atomic formulas out of terms, equations, and relations symbols. Then, for the fragment of logic we consider in this paper, we allow to take conjunctions and existential quantifications on these. For any $\LL$-structure $A$ and formula $\varphi(x)$ of arity $X$, we define its interpretation as an $\cm$-subobject
\begin{center}
	\begin{tikzpicture}[baseline=(current  bounding  box.south), scale=2]

		\node (d0) at (0,0) {$\varphi_A$};
		\node (c1) at (1,0) {$ A^X;$};
		
		\path[font=\scriptsize]
		
		(d0) edge [>->] node [above] {$\cm$} (c1);
	\end{tikzpicture}	
\end{center}
this is done recursively on the construction of $\varphi$. 

The chosen factorization system is particularly important for the interpretation of existential quantification: given a formula $\psi(x,y)$ of arity $X+Y$, we define the interpretation of $\varphi(x):=(\exists y)\psi(x,y)$ as the $(\ce,\cm)$-factorization below.
\begin{center}
	\begin{tikzpicture}[baseline=(current  bounding  box.south), scale=2]
		
		\node (a0) at (-0.2,0.8) {$\psi_A$};
		\node (c0) at (1,0.8) {$ A^{X}\times A^Y $};
		\node (c1) at (2.3,0.8) {$ A^X$};
		\node (d0) at (1,1.4) {$\varphi_A$};
		
		\path[font=\scriptsize]
		
		(a0) edge [>->] node [below] {} (c0)
		(c0) edge [->] node [below] {$p_1$} (c1)
		(a0) edge [->>] node [above] {$\ce\ \ \ $} (d0)
		(d0) edge [>->] node [above] {$\ \ \ \cm$} (c1);
	\end{tikzpicture}	
\end{center}
This generalizes the interpretation of existential quantification within a topos \cite{MLM}. In fact, in the case where $\cv$ is a Grothendieck topos with its cartesian closed structure and equipped with the (epi, mono) factorization system, our enriched logic extends the internal logic of the topos to the case where arities are non-discrete (meaning that they are not coproducts of the terminal object in general).

Then we define sequents
$$
(\forall x)(  \varphi(x) \vdash \psi(x)),
$$
of formulas of the same arity $X$, and say that an $\LL$-structure $A$ satisfies them if $\varphi_A\subseteq\psi_A$ as $\cm$-subobjects of $A^X$. Ordinarily, regular theories are defined as set of sequents as above where $\varphi$ is a conjunction of atomic formulas and $\psi$ is a  positive-primitive formula; that is, obtained by applying existential quantification to conjunctions of atomic formulas. Equivalently, one can take also $\varphi$ to be positive-primitive without changing the power of the fragment; both these notions have occurred in the literature. In the finitary setting, models of such regular theories are known to characterize finite injectivity classes in locally finitely presentable categories~\cite[Chapter~4]{AR}, as well as those full subcategories that are closed under products, filtered colimits, and pure subobjects~\cite{RAB02}.

In the enriched context, injectivity classes were first introduced in~\cite{LR12} and studied for instance in \cite{AR1,BLV,LT24}. The notion relies on a class of maps $\ce$ generalizing that of surjections; in our case this corresponds to the left class of the chosen factorization system on $\cv$. Then, more recently, we introduced an enriched notion of purity~\cite{RTe1} depending on a factorization system on $\cv$, as in the current setting. The aim was to characterize enriched $\ce$-injectivity classes as those closed under a certain class of limits, filtered colimits, and $\ce$-pure subobjects \cite[Theorem~5.5]{RTe1}.

In this paper we shall see that, under some hypothesis on $(\ce,\cm)$ and on the given language $\LL$, the $\ce$-pure morphism from \cite{RTe1} in the $\cv$-category $\Str(\LL)$ can be characterized in a model theoretic way as those homomorphisms of structures $f\colon K\to L$ with are elementary with respect to every positive-primitive formula $\varphi$ (Section~\ref{purity-section}). 

After this, we complete the connection with the enriched notion of injectivity, by defining our {\em enriched regular theories} to be sets of sequents of the form
$$(\forall x)(  \varphi(x) \vdash (\exists y)(\psi(x,y)\wedge \varphi(x)))$$
where $\varphi$ and $\psi$ are conjunctions of atomic formulas (Definition~\ref{regular}). This slightly unusual choice is dictated from the fact that, unlike in the ordinary case or the case internal to a topos,  our class $\ce$ may not be stable under pullbacks; making certain deduction rules of ordinary regular logic fail. In particular, the sequent above is not equivalent to the same sequent where $\varphi$ is omitted from the right-hand-side (Remark~\ref{existential-wedge}).

Then we can prove:

\begin{theo:main}
	Under Assumption~\ref{assumption}, the following are equivalent for a full subcategory $\ca$ of $\Str(\LL)$:\begin{enumerate}
		\item $\ca$ is a $(\lambda,\ce)$-injectivity class in $\Str(\LL)$;
		\item $\ca\cong\Mod(\TT)$ for a regular $\mathbb L_{\lambda\lambda}$-theory $\TT$.
	\end{enumerate}
\end{theo:main}

The key tool for capturing $\ce$-injectivity in the theorem above, is proving the existence of {\em presentation formulas} for given $\LL$-structures (see \cite[Remark~5.5]{AR} for the ordinary version): we say that a conjunction of atomic formulas $\varphi$ is a presentation formula for $A\in\Str(\LL)$ if homming out of $A$ is isomorphic to the evaluation $\cv$-functor
$$\varphi_{(-)}\colon\Str(\LL)\longrightarrow\cv.$$
The fact that these exist for any $A\in\Str(\LL)$ under Assumption~\ref{assumption}, is discussed in Section~\ref{sect:pres}.

When we are in the context of~\cite{RTe1}, so that $\ce$-injectivity classes can be characterized under certain closure properties, then the theorem above implies the following.

\begin{theo1}[part of Theorem~\ref{inj2}]
	The following are equivalent for a full subcategory $\ca$ of $\Str(\LL)$:\begin{enumerate}
		\item $\ca\cong\Mod(\TT)$ for a regular $\mathbb L_{\lambda\lambda}$-theory $\TT$;
		\item $\ca\cong\Mod(\TT)$ for a theory $\TT$ with sequents of the form
		$$(\forall x)(  \varphi(x) \vdash \psi(x))$$
		where $\varphi$ and $\psi$ are positive-primitive $\mathbb L_{\lambda\lambda}$-formulas;
		\item $\ca$ is closed under products, powers by $\ce$-stable objects, $\lambda$-filtered colimits, and $\lambda$-elementary subobjects.
	\end{enumerate}
\end{theo1}

The $\lambda$-elementary morphisms are defined in Section~\ref{purity-section}, and are a natural generalization of the ordinary pure ones to the enriched context. It follows from the theorem that, in this case, one can still consider the more traditional kind of regular theories (with positive-primitive formulas on both sides of the sequents) and at the same time characterize models of such theories in terms of closure under specific limits, filtered colimits, and elementary embeddings.

Beside Grothendieck toposes, our examples of bases of enrichment include the symmetric monoidal closed categories $\Met$ of generalized metric spaces with non-expanding maps, $\Ban$ of Banach spaces and linear maps of norm $\leq 1$ with the factorization systems (dense, closed isometry) and (dense, isometry) respectively. In the second case, our enriched logic is related to the logic of positive bounded formulas and approximate semantics~\cite{I} (see Example~\ref{Iovino}).

Other cases of particular interest that are covered by our theory, are those bases of enrichment for which $\cv_0$ is a quasivariety~\cite{AR} equipped with the (regular epi, mono) factorization system. Examples include the category $\Ab$ of abelian groups, $\mathbb R$-$\Mod$ of modules over a ring $R$, and $\bo{DGAb}$ of differentially graded abelian groups.

We shall also discuss the case where $\cv=\Cat$ with the (surjective on objects, injective on objects and fully faithful) factorization system. In this setting, we can describe sketches, categories with limits of some type, as well as categories equipped with an abelian group bundle and small tangent categories (see \cite{R4,CC}).

\section{Preliminaries}

As in the setting of enriched universal algebra~\cite{RTe}, our base of enrichment is a symmetric monoidal closed category $\cv=(\cv_0,\otimes,I)$ which is locally $\lambda$-presentable as a closed category \cite{K}, for some $\lambda$. This means that $\cv_0$ is locally $\lambda$-presentable and the full subcategory $\cv_\lambda$ spanned by the $\lambda$-presentable objects contains the unit and is closed under the monoidal structure of $\cv_0$. As in \cite{RTe}, for notational simplicity we denote by 
$$A^X:=[X,A]$$ 
the internal hom in $\cv$. For any set $X$, coproducts $X\cdot I:=\sum_{x\in X}I$ of copies of $I$ are called \textit{discrete objects} of $\cv$. Note that, for every object $V$ of $\cv$ there is the induced morphism $\delta_V:V_0\to V$ where $V_0=\cv_0(I,V)\cdot I$. 

We will make use of the enriched notions of weighted limits and colimits, for which we direct the reader to~\cite{K}. In particular we will mostly consider conical limits and colimits, as well as powers and copowers (which together generate all weighted limits and colimits).

Finally, we assume our base of enrichment $\cv$ to come equipped with an {\em enriched (orthogonal) factorization system} $(\ce,\cm)$ in the sense of \cite{LW}; this means that $(\ce,\cm)$ is an ordinary factorization system on $\cv_0$ for which the class $\ce$ is closed in the arrow $\cv$-category $\cv^\to$ under copowers (or equivalently, $\cm$ is closed under powers). We shall need the following result about such factorization systems:

\begin{lemma}\label{M-lp}
	If $\cm$ is closed in $\cv^\to$ under $\lambda$-filtered colimits, then $\cm$ is locally $\lambda$-presentable as a $\cv$-category and the inclusion $J\colon\cm\hookrightarrow\cv^\to$ has a left adjoint $L\colon\cv^\to\to\cm$ given pointwise by sending $f\colon X\to Y$ to the $\cm$-subobject $Lf\colon LX\rightarrowtail Y$ induced by the $(\ce,\cm)$-factorization 
	$$ X\stackrel{\ce}{\twoheadrightarrow} LX\stackrel{\cm}{\rightarrowtail} Y $$
	of the map $f$.
\end{lemma}
\begin{proof}
	The $\cv$-category of arrows $\cv^\to$ is locally $\lambda$-presentable because $\cv$ is. Moreover $\cm$ is closed in $\cv^\to$ under all conical limits and powers (since $(\ce,\cm)$ is enriched); thus is closed under all weighted limits. By \cite[Theorem~2.48]{AR} $J_0$ has a left adjoint, and, since $J$ preserves powers, the ordinary left adjoint is actually an enriched one. Thus $\cm$ is locally $\lambda$-presentable. It is easy to see that the left adjoint acts as described in the statement.
\end{proof}

From Section~\ref{sect:pres} this factorization system will be assumed to be \textit{proper}; that is, we will assume every element of $\ce$ to be an epimorphism and every element of $\cm$ a monomorphism. It will also be useful to recall the following notation from \cite{RTe1}:

\begin{defi}
	An object $X\in\cv$ is called {\em $\ce$-projective} if $\cv_0(X,-)\colon\cv_0\to\Set$ sends maps in $\ce$ to epimorphisms. While, $X$ is called {\em $\ce$-stable} if $e^X\colon A^X\to B^X$ is in $\ce$ whenever $e\colon A\to B$ is. 
\end{defi}
 
It is easy to see that if the unit is $\ce$-projective, every $\ce$-stable object is $\ce$-projective. 

A morphism $f\colon A\to B$ in $\cv$ is called a \textit{surjection} if $\cv_0(I,f)$ is surjective. A morphism $g\colon C\to D$ is called an \textit{injection} if it has the unique right lifting property to all surjections morphisms (see \cite{R1}). In what follows, $Surj$ will denote the class of surjections and $Inj$ the class of injections.

\begin{exams}\label{bases}
	We now give a list of bases of enrichment, together with an enriched factorization system, that will be relevant to the setting of the paper. 
{\setlength{\leftmargini}{1.6em}
\begin{enumerate}
	\item $\cv=\Met$ where distances $\infty$ are allowed. The resulting category is locally $\aleph_1$-presentable as a closed category and symmetric monoidal closed (where $A\otimes B$ is $A\times B$ with the metric $d((a,b),(a,b'))=d(a,a')+d(b,b')$). Discrete objects are discrete metric spaces (because $I$ is the one-point metric space). As factorization systems we can take either (surjective, isometry) or (dense, closed isometry) (see \cite{AR1}). Both factorization systems are enriched and proper. In the first case, $I$ is $\ce$-projective while in the second case it is not $\ce$-projective. In both cases, $\ce$ is closed under products in $\Met^\to$ and discrete objects are $\ce$-stable. Note that the first factorization system is $(Surj,Inj)$. We can also consider $\cv$ to be the category $\CMet$ of complete metric spaces. Here we only consider one factorization system given by (dense, isometry) (see \cite{AR1}).  
	\item $\cv=\Ban$ the category of Banach spaces with linear maps of norm $\leq 1$. This is symmetric monoidal closed where $\otimes$
	is the projective tensor product, and internal hom is the space $[K,L]$ consisting of \textit{all} bounded linear mappings (not necessarily of norm at most 1) from $K$ to $L$ (see \cite[Example~6.1.9h]{B}). Observe that $\Ban_0(K,L)$ is the unit ball of $[K,L]$; moreover, $\Ban$ is locally $\aleph_1$-presentable (\cite[Example~1.48]{AR}). 
	The factorization system (dense, isometry) corresponds to the (epi, strong mono) factorization system, and hence it is proper. It is also enriched because isometries are closed under powers.
	The (strong epimorphism, monomorphism) factorization is described in \cite[Section~1]{Po}; monomorphisms are one-to-one maps and strong epimorphisms (coincide with regular epimorphisms and) are quotient maps $X\to X/Y$ where $Y$ is a closed subspace of $X$. This factorization system is also enriched. Moreover, strong epimorphisms coincide with surjections \cite[Example~3.5(8)]{RTe1}. Discrete objects are coproducts of $\mathbb C$ (because $I=\mathbb C$) and they are strong epi-projective and strong epi-stable. The unit ball functor $\Ban_0(\mathbb C,-)\colon \Ban\to\Set$ has a left adjoint $l_1$ sending $X$ to the discrete object $X\cdot\mathbb C$. 
	
	\item $\cv=\Cat$ is cartesian closed and locally finitely presentable.
	The factorization system (surjective on objects, injective on objects fully faithful) coincides with $(Surj,Inj)$. Note that this factorization system is enriched but not proper (because $1$ is not a generator). Discrete objects are discrete categories and they are $\ce$-stable. 
	
	\item $\cv=\Pos$ the cartesian closed category of posets and monotone maps, which is locally finitely presentable as a closed category. This comes equipped with the proper enriched factorization system given by $(Surj,Inj)$. Discrete objects are $\ce$-stable and $\ce$-projective.
	
	\item We can consider $\cv$ to be any regular base of enrichment with the (regular epi, mono) factorization system, which is enriched and proper. The $\ce$-projective objects are the usual regular projectives; these are also $\ce$-stable if in addition $\cv$ is a symmetric monoidal quasivariety as in~\cite{LT20}. Examples of such a $\cv$ include the category $\Ab$ of abelian groups, $R$-$\Mod$ of modules of a ring $R$,  $\bo{GAb}$ of graded abelian groups, and $\bo{DGAb}$ of differentially graded abelian groups.
\end{enumerate}}
Within the last example falls also the case where $\cv$ is a Grothendieck topos with its cartesian closed structure. The factorization system corresponds to the (epi, mono) one; this is the only proper factorization system in a topos (since every epimorphism is regular). The notion of regular logic we consider in this paper will extend that of (single-sorted) regular logic internal to the topos~\cite{MLM}, by allowing the use of non-discrete arities.
\end{exams}

\begin{rem}
	We expect most of the results of this paper to generalize to the setting of~\cite{LWP} where the base of enrichment $\cv$ is assumed to be {\em locally bounded} instead of locally presentable. This allows to capture other important examples of enrichment that have a more topological flavour (such as enrichment over compactly generated topological spaces). Nonetheless, we one still needs $\cv$ to be monoidal {\em closed}, as this is fundamental for the interpretation of our enriched arities (see for instance Definition~\ref{structure}). We decided to restrict to the locally presentable case as most of the results we rely on (such as \cite{RTe,RTe1,LR12}) assume that.
\end{rem}

\section{Enriched languages}

In this section we introduce the notion of language that will be central in the development of our theory. This generalizes the concept we considered in \cite{RTe} by allowing the interpretation of relation symbols. Before proceeding, we set the following assumption which is satisfied by all the examples mentioned in the preliminaries.

\begin{assume}\label{assum-0}
	We fix an enriched factorization system $(\ce,\cm)$ on $\cv$ whose right class $\cm$ is made of monomorphisms and is closed in $\cv^\to$ under $\lambda$-filtered colimits.
\end{assume}

\begin{rem}
	Note that we need to assume $\cm$ to be contained in the class of monomorphisms of $\cv$ as this will imply the property of Remark~\ref{unique-restriction} below and will be needed in the proofs of Proposition~\ref{lambda-pres} and Lemma~\ref{strong-pres}.
\end{rem}

As mentioned above, the following definition extends the notion of functional languages from~\cite{RTe} to the case where also relation symbols are allowed:

\begin{defi}
	A (single-sorted) {\em language} $\mathbb L$ (over $\cv$) is the data of:\begin{enumerate}
		\item a set of function symbols $f\colon(X,Y)$ whose arities $X$ and $Y$ are objects of $\cv$;
		\item a set of relation symbols $R:X$, with arity an object $X$ of $\cv$.
	\end{enumerate}
	The language $\mathbb L$ is called {\em $\lambda$-ary} if all the arities appearing in $\mathbb L$ lie in $\cv_\lambda$.
\end{defi}

For every language $\LL$ there is an associated notion of $\LL$-structure:

\begin{defi}\label{structure}
	Given a language $\mathbb L$, an {\em $\mathbb L$-structure} is the data of an object $A\in\cv$ together with:\begin{enumerate}
		\item a morphism $f_A\colon A^X\to A^Y$ in $\cv$ for any function symbol $f\colon(X,Y)$ in $\mathbb L$;
		\item an $\cm$-subobject $r_A\colon R_A\rightarrowtail A^X$ for any relation symbol $R\colon X$ in $\mathbb L$.
	\end{enumerate}
	A {\em morphism of $\mathbb L$-structures} $h\colon A\to B$ is determined by a map $h\colon A\to B$ in $\cv$ making the following square commute
	\begin{center}
		\begin{tikzpicture}[baseline=(current  bounding  box.south), scale=2]
			
			\node (a0) at (0,0.8) {$A^X$};
			\node (b0) at (1,0.8) {$B^X$};
			\node (c0) at (0,0) {$A^Y$};
			\node (d0) at (1,0) {$B^Y$};
			
			\path[font=\scriptsize]
			
			(a0) edge [->] node [above] {$h^X$} (b0)
			(a0) edge [->] node [left] {$f_A$} (c0)
			(b0) edge [->] node [right] {$f_B$} (d0)
			(c0) edge [->] node [below] {$h^Y$} (d0);
		\end{tikzpicture}	
	\end{center} 
	for any $f\colon(X,Y)$ in $\mathbb L$, and such that there exist a map $h_R\colon R_A\to R_B$ completing the diagram below
	\begin{center}
		\begin{tikzpicture}[baseline=(current  bounding  box.south), scale=2]
			
			\node (a0) at (0,0.8) {$R_A$};
			\node (b0) at (1,0.8) {$R_B$};
			\node (c0) at (0,0) {$A^X$};
			\node (d0) at (1,0) {$B^X$};
			
			\path[font=\scriptsize]
			
			(a0) edge [dashed, ->] node [above] {$h_R$} (b0)
			(a0) edge [>->] node [left] {$r_A$} (c0)
			(b0) edge [>->] node [right] {$r_B$} (d0)
			(c0) edge [->] node [below] {$h^X$} (d0);
		\end{tikzpicture}	
	\end{center} 
	for any relation symbol $R:X$ in $\mathbb L$.
\end{defi}

\begin{rem}\label{unique-restriction}
	Since $\cm$ is made of monomorphisms, every map $h_R$ above is uniquely determined by the morphism $h$. 
\end{rem}

So far $\mathbb{L}$-structures and morphisms between them form just an ordinary category $\Str(\LL)_0$. As we did in~\cite{RTe}, we define a $\cv$-category $\Str(\LL)$ whose underlying ordinary category (as the name suggests) will be the one just introduced. This will require some steps. 

Note that, given a $\lambda$-ary language $\mathbb L$, we can see it as $\mathbb L=\mathbb F\cup\mathbb R$ the union of the $\lambda$-ary language $\mathbb F$, with only the function symbols of $\mathbb L$, and the $\lambda$-ary language $\mathbb R$, with only the relation symbols of $\mathbb L$. Then, the $\cv$-category $\Str(\FF)$ of $\FF$-structures was constructed in \cite[Section~3]{RTe}. The next step is to construct $\Str(\RR)$ for the relational part of the language $\mathbb R$. 

Let $\mathbb R_\cv$ the discrete $\cv$-category on the set $\mathbb R$, and let $2_\cv$ be the free $\cv$-category on the category $2=\{0\to 1\}$. We can consider the $\cv$-category $$2_\cv\otimes\mathbb R_\cv\cong \textstyle\sum_{\mathbb R}2_\cv$$ and the pushout in $\cv$-$\Cat$
\begin{center}
	\begin{tikzpicture}[baseline=(current  bounding  box.south), scale=2]
		
		\node (a0) at (0,0.8) {$\mathbb R_\cv$};
		\node (b0) at (1.1,0.8) {$2_\cv\otimes\mathbb R_\cv$};
		\node (c0) at (0,0) {$\cv_\lambda^{op}$};
		\node (d0') at (0.92,0.15) {$\ulcorner$};
		\node (d0) at (1.1,0) {$\Theta_\mathbb R$};
		
		\path[font=\scriptsize]
		
		(a0) edge [->] node [above] {$j$} (b0)
		(a0) edge [->] node [left] {$i$} (c0)
		(b0) edge [->] node [right] {$H_\mathbb R$} (d0)
		(c0) edge [->] node [below] {$\theta_\mathbb R$} (d0);
	\end{tikzpicture}	
\end{center} 
where $i(R:X)=X$ and $j(R:X)=(1,R:X)$. Before defining the $\cv$-category of $\mathbb R$-structures, note that we have isomorphisms
$$ [2_\cv\otimes\mathbb R_\cv,\cv]\cong \prod_\mathbb R [2_\cv,\cv]\cong \prod_\mathbb R\cv^\to. $$
Thus we can construct $\Str(\RR)$ as follows.

\begin{defi}
	The $\cv$-category $\Str(\RR)$ on a $\lambda$-ary language $\mathbb R$ of relation symbols is defined as the pullback
	\begin{center}
		\begin{tikzpicture}[baseline=(current  bounding  box.south), scale=2]
			
			\node (a0) at (0,0.9) {$\Str(\RR)$};
			\node (a0') at (0.3,0.7) {$\lrcorner$};
			\node (b0) at (2.2,0.9) {$[\Theta_\mathbb R,\cv]$};
			\node (c0) at (0,0) {$\cv\times \prod_\mathbb R\cm$};
			\node (d0) at (2.2,0) {$[\cv_\lambda^{op},\cv]\times \prod_\mathbb R\cv^\to$};
			
			\path[font=\scriptsize]
			
			(a0) edge [right hook->] node [above] {$V_\mathbb R$} (b0)
			(a0) edge [->] node [left] {$U_\mathbb R=(U_\mathbb R^1,U_\mathbb R^2)$} (c0)
			(b0) edge [->] node [right] {$([\theta_\mathbb R,\cv],[H_\mathbb R,\cv]')$} (d0)
			(c0) edge [right hook->] node [below] {$N\times \prod_\mathbb R J$} (d0);
		\end{tikzpicture}	
	\end{center} 
	where $[H_{\mathbb R},\cv]'$ is the composite of $[H_{\mathbb R},\cv]$ with the isomorphisms given above, $N=\cv(H_{\mathbb R}-,1)$ is fully faithful, and $J\colon \cm\hookrightarrow\cv^\to$ is the inclusion.
\end{defi}

Arguing as in \cite[Remark~3.4]{RTe} one can show that the definition of $\Str(\RR)$ does not depend on the choice of $\lambda$ that makes $\RR$ a $\lambda$-ary language.

\begin{propo}\label{R-str}
Let $\mathbb R$ be a $\lambda$-ary language of relation symbols; then:\begin{enumerate}
		\item the underlying category of $\Str(\RR)$ has $\mathbb R$-structures as objects and maps of $\mathbb R$-structures as morphisms;
		\item $\Str(\RR)$ is locally $\lambda$-presentable as a $\cv$-category;
		\item $U_\mathbb R\colon \Str(\RR)\to \cv\times \prod_\mathbb R\cm$ is a conservative right adjoint and preserves $\lambda$-filtered colimits.
	\end{enumerate}
\end{propo}
\begin{proof}
	$(1)$ By construction, an object of $\Str(\RR)$ is a $\cv$-functor $\Theta_\mathbb R\to\cv$ whose restriction along $\theta_\mathbb R$ is isomorphic to $A^{(H_{\mathbb R}-)}$ for some $A\in\cv$, and whose restriction along $H_\mathbb R$ lands pointwise in $\cm$. Now, by definition of $\theta_\mathbb R$, to give the data above is equivalent to give an object $A\in\cv$ together with an ordinary functor $2\times\mathbb R\to \cv_0$ which sends $(1,R:X)$ to $A^X$, for any $(R:X)\in\mathbb R$, and sends every arrow in $2\times\mathbb R$ to a map in $\cm$. In other words, it is equivalent to give an object $A\in\cv$ together with an assignment $\mathbb R\to \cm$ sending a relation $R:X$ to a map in $\cm$ with codomain $A^X$. This is exactly the data of an $\mathbb R$-structure. The same argument applies to morphisms of the underlying category. 
	
	$(2)$ The pullback defining $\Str(\RR)$ is a bipullback since $([\theta_\mathbb R,\cv],[H_\mathbb R,\cv]')$ is an isofibration. Moreover, the $\cv$-functors $([\theta_\mathbb R,\cv],[H_\mathbb R,\cv]')$ and $N\times \prod_\mathbb R J$ are both continuous and $\lambda$-filtered colimit preserving (by the lemma above), and the $\cv$-categories involved are locally $\lambda$-presentable (for $\cm$ this follows from Lemma~\ref{M-lp}, then use stability of locally $\lambda$-presentable $\cv$-categories under products~\cite{Bird}). Thus $\Str(\RR)$ is locally $\lambda$-presentable (again by~\cite{Bird}), and both $U_\mathbb R$ and $V_\mathbb R$ preserve all limits and $\lambda$-filtered colimits.
	
	$(3)$ We have already shown that $U_\mathbb R$ is continuous and preserves $\lambda$-filtered colimits. Since $\cv\times \prod_\mathbb R\cm$ is locally presentable, then it also has a left adjoint. The fact that $U_\mathbb R$ is conservative follows from $(1)$ and the definition of $\mathbb R$-structure.
\end{proof}

 We can now define the $\cv$-category of $\mathbb L$-structures for a general $\lambda$-ary language $\mathbb L$. 

\begin{defi}
	The $\cv$-category $\Str(\LL)$ on a $\lambda$-ary language $\mathbb L$ is defined as the pullback
	\begin{center}
		\begin{tikzpicture}[baseline=(current  bounding  box.south), scale=2]
			
			\node (a0) at (0,0.9) {$\Str(\LL)$};
			\node (a0') at (0.3,0.7) {$\lrcorner$};
			\node (b0) at (1.4,0.9) {$\Str(\FF)$};
			\node (c0) at (0,0) {$\Str(\RR)$};
			\node (d0) at (1.4,0) {$\cv$};
			
			\path[font=\scriptsize]
			
			(a0) edge [->] node [above] {$J_\mathbb F$} (b0)
			(a0) edge [->] node [left] {$J_\mathbb R$} (c0)
			(b0) edge [->] node [right] {$U_\mathbb F$} (d0)
			(c0) edge [->] node [below] {$U_\mathbb R^1$} (d0);
		\end{tikzpicture}	
	\end{center} 
	where $\mathbb L=\mathbb F\cup\mathbb R$ has been written as the union of its sub-languages of function and relation symbols respectively. We denote by $U\colon \Str(\LL)\to \cv$ the diagonal of the square.
\end{defi}

Below we denote by $V\colon \Str(\LL)\to \prod_\mathbb R\cm$ the composite $U_\mathbb R^2\circ J_\mathbb R$ sending an $\mathbb L$-structure $A$ to the family $(R_A\rightarrowtail A^X)_{(R:X)}$.

\begin{theo}\label{L-str}
	 Let $\mathbb L=\mathbb F\cup\mathbb R$ be a $\lambda$-ary language; then:
\begin{enumerate}
		\item the underlying category of $\Str(\LL)$ has $\mathbb L$-structures as objects and maps of $\mathbb L$-structures as morphisms;
		\item the $\cv$-category $\Str(\LL)$ is locally $\lambda$-presentable;
		\item the forgetful $\cv$-functor $U\colon \Str(\LL)\to \cv$ is a right adjoint and preserves $\lambda$-filtered colimits;
		\item the pair $(U,V)\colon \Str(\LL)\to \cv\times \prod_\mathbb R\cm$ is a conservative right adjoint which preserves $\lambda$-filtered colimits.
	\end{enumerate}
\end{theo}
\begin{proof}
	This is obtained simply by combining \cite[Proposition~3.5]{RTe} and \ref{R-str} and using the stability of locally $\lambda$-presentable categories under bilimits~\cite{Bird} (since $U_\mathbb F$ and $U_\mathbb R$ are isofibrations, the pullback above is also a bipullback).
\end{proof}

Let us show an example of $\cv$-category of structures:

\begin{exam}\label{Sketches}
	Let $\cv=\bo{Cat}$ with the (surjective on objects, injective on objects fully faithful) factorization system. We define a relational language for the 2-category of sketches with specified shape of cones and cocones.
	
	Given a $C\in\Cat$, denote by $0*C$ the category obtained by adding to $C$ an initial object $0$, and let $k\colon C\to 0*C$ be the inclusion. Then, given another category $A$, an element $x\in A^{0*C}$ is a cone over some functor $C\to A$, while an element $y\in A^C$ is just a functor $C\to A$. Dually, we can define $C*1$ by adding a terminal object, so that an element $x\in A^{C*1}$ is a cocone over some functor $C\to A$.
	
	Given small collections $S$ and $T$ of categories, consider the relational language $\RR_{S,T}$ consisting of a relation symbol $R_C:0*C$ for any $C\in S$ and a relation symbol $R^D: D*1$ for any $D\in T$. Then, its is easy to see that an $\RR_{S,T}$-structure is the data of a category $A$ together with a set of cones of shape $C$ for each $C\in S$, and of cocones of shape $D\in T$. This is the same as a {\em sketch} whose specified cones and cocones have shape in $S$ and $T$ respectively. Similarly, a morphism of $\RR_{S,T}$-structures is the same as a morphism of sketches: a functor $A\to B$ that maps specified co/cones in $A$ into specified co/cones in $B$.
	
	As a consequence $\Str(\RR_{S,T})\cong \tx{Skt}_{S,T}$ is the 2-category of sketches with cones of shape in $S$ and cocones of shape in $T$; the 2-cells in $\tx{Skt}_{S,T}$ are just natural transformations between the underlying functor of a morphism of sketches. By Theorem~\ref{L-str} above this is locally presentable, and in particular complete and cocomplete. 
\end{exam}

We conclude this section by characterizing the $\lambda$-presentable objects of $\Str(\LL)$. We shall need first the following result:

\begin{propo}
	Given a language $\mathbb L=\mathbb F\cup\mathbb R$, the $\cv$-functor 
$$
J_\mathbb F\colon\Str(\LL)\to \Str(\FF)
$$ 
has both a left and a right adjoint which can be described as follows:
{\setlength{\leftmargini}{1.6em}
\begin{itemize}
		\item The left adjoint $F\colon\Str(\FF)\to \Str(\LL)$ takes an $\mathbb F$-structure $A$ to the $\mathbb L$-structure obtained from $A$ by interpreting the relations $R:X$ in $\mathbb R$ with the $\cm$-subobject
		$$ F0\rightarrowtail A^X $$
		induced by the $(\ce,\cm)$-factorization of $!\colon 0\to A^X$.
		\item The right adjoint $R\colon\Str(\FF)\to \Str(\LL)$ takes an $\mathbb F$-structure $A$ to the $\mathbb L$-structure obtained from $A$ by interpreting $R:X$ in $\mathbb R$ with the $\cm$-subobject
		$$ 1\colon A^X\rightarrowtail A^X $$
		given by the identity on $A^X$.
	\end{itemize}}
	Both $F$ and $R$ are fully faithful.
\end{propo}
\begin{proof}
	It is easy to see that $F$ and $R$ define respectively left and right adjoints to the underlying ordinary functor of $J_\mathbb F$, and that they are fully faithful as ordinary functors. To conclude, by \cite[Theorem~4.85]{Kel82}, it is enough to prove that $J_\mathbb F$ preserves powers and copowers by elements of $\cv$. The former follows by definition since $J_\mathbb F$ is obtained by pulling back continuous functors. For the latter, consider an $\mathbb L$-structure $A$ and some $Y\in\cv$; then the copower $X\cdot A$ of $A$ by $X$ can be constructed by considering the copower $X\cdot J_\mathbb FA$ in $\Str(\FF)$, and by extending it to an $\mathbb L$-structure as follows: for any relation $R:Y$ we define $r_{X\cdot A}\colon R_{X\cdot A}\rightarrowtail (X\cdot A)^Y$ as the $\cm$-subobject obtained by taking the $(\ce,\cm)$-factorization of the composite
	$$ X\cdot R_A \xrightarrow{X\cdot r_A} X\cdot (A^Y)\longrightarrow (X\cdot A)^Y $$
	in $\cv$. It follows that $J_\mathbb F$ also preserves copowers.
\end{proof}

In the next proposition we characterize the $\lambda$-presentable objects of $\Str(\LL)$. We shall need the following definition:

\begin{defi}
	We say that an element $A\in\cv$ is an $\ce$-\textit{quotient} of $Y\in\cv$ if there exists a map $Y\twoheadrightarrow A$ in $\ce$.
\end{defi}

Note that when $m\colon A\rightarrowtail B$ is in $\cm$ and $A$ is an $\ce$-quotient of the initial object $0$, then $m$ is induced by the $(\ce,\cm)$-factorization of the unique map $!\colon 0\to B$.

\begin{propo}\label{lambda-pres}
	An $\mathbb L$-structure $A$ is $\lambda$-presentable in $\Str(\LL)$ if and only if:
	{\setlength{\leftmargini}{2em}
		\begin{enumerate}
		\item[(i)] $J_{\mathbb{F}}(A)$ is $\lambda$-presentable in $\Str(\FF)$;
		\item[(ii)] for every $R: X$ in $\LL$, its interpretation $R_A$ is an $\ce$-quotient of some $Y\in\cv_\lambda$;
		\item[(iii)] all but a $\lambda$-small number of the $R_A$, for $R:X$ in $\mathbb L$, are $\ce$-quotients of $0$.
	\end{enumerate}}
\end{propo}
\begin{proof}
	It is enough to work at the level of the underlying ordinary category $\Str(\LL)_0$ since enriched and ordinary finitely presentable objects coincide~\cite[Proposition~7.5]{Kel82}. Let $\cg$ be the full subcategory of $\Str(\LL)_0$ with objects those $A$ satisfying the three properties above. To conclude it is enough to show (1) that every object of $\cg$ is $\lambda$-presentable, (2) that $\cg$ is closed under $\lambda$-small colimits in $\Str(\LL)_0$, and (3) that $\cg$ is strongly generating.
	
	(1). Let $A\in\cg$ and $h\colon A\to L:=\colim_iL_i$ be a morphism into a filtered colimit of objects $L_i$ in $\Str(\LL)_0$, denote the colimiting cone by $(c_i\colon L_i\to L)_i$. Since $J_{\mathbb{F}}(A)$ is $\lambda$-presentable and $J_{\mathbb{F}}$ preserves $\lambda$-filtered colimits, $J_{\mathbb{F}}(h)$ factors through some $J_{\mathbb{F}}(c_i)$ as a morphism $k\colon J_{\mathbb{F}}(A)\to J_{\mathbb{F}}(L_i)$ of $\mathbb F$-structures. We need to extend $k$ to a map of $\mathbb L$-structures, maybe changing index $i$.\\
	Going back in $\Str(\LL)_0$, for every relation $R:X$ we have the solid part of diagram below
	\begin{center}
		\begin{tikzpicture}[baseline=(current  bounding  box.south), scale=2]
			
			\node (0) at (-2,0.8) {$Y$};
			\node (a) at (-1,0.8) {$R_A$};
			\node (a0) at (0,0.8) {$R_{L_i}$};
			\node (b0) at (1,0.8) {$R_L$};
			\node (b) at (-1,0) {$A^X$};
			\node (c0) at (0,0) {$L_i^X$};
			\node (d0) at (1,0) {$L^X$};
			
			\path[font=\scriptsize]
			
			(0) edge [->>] node [above] {$q$} (a)
			(a) edge [dashed, ->] node [below] {$k_R$} (a0)
			(a) edge [bend left,->] node [above] {$h_R$} (b0)
			(a) edge [>->] node [left] {$r_A$} (b)
			(b) edge [->] node [below] {$k^X$} (c0)
			
			(a0) edge [->] node [below] {$(c_i)_R$} (b0)
			(a0) edge [>->] node [right] {$r_{L_i}$} (c0)
			(b0) edge [>->] node [right] {$r_L$} (d0)
			(c0) edge [->] node [below] {$c_i^X$} (d0);
		\end{tikzpicture}	
	\end{center}
	with $Y\in\cv_\lambda$. Now, if $Y=0$, then $k_R$ is the unique map into $R_{L_i}$ induced by $!\colon 0\to R_{L_i}$ and the orthogonality property of the factorization system. Otherwise, since $R_L$ is the $\lambda$-filtered colimit of the $R_{L_i}$, the map $h_Rq$ factors through some $(c_j)_R\colon R_{L_j}\to R_L$. Without loss of generality we can assume $i=j$ and hence we have a map $s\colon Y\to R_{L_i}$ such that $ r_{L_i}s= (k^Xr_A)q$. Since $q$ is in $\ce$ and and $r_{L_i}$ in $\cm$, there is an induced $k_R\colon R_A\to R_{L_i}$ making the square in the diagram above commute. Since there are only a $\lambda$-small number relations which are not $\ce$-quotients of $0$, we can repeat this argument and obtain a map $k\colon A\to L_i$ in $\Str(\LL)_0$ such that $c_ik=h$. By construction, any two such factorizations $k\colon A\to L_i$ and $k'\colon A\to L_j$ of $h$ must coincide after composition with some maps $i\to l$ and $j\to l$. This proves that $A$ is $\lambda$-presentable in $\Str(\LL)_0$.
	
	(2). Since $J_\mathbb F$ is cocontinuous, the objects of $\Str(\LL)_0$ satisfying the first condition are clearly closed under $\lambda$-small colimits. \\
	In general, given a $\lambda$-small limit $A:=\colim_i(A_i)$ in $\Str(\LL)_0$ and a relation symbol $R:X$ in $\mathbb L$, the $\cm$-subobject $R_A$ of $A^X$ is determined by the $(\ce,\cm)$-factorization of the induced map $\colim_i(R_{A_i})\to A^X$. It follows that $R_A$ is an $\ce$-quotient of $\colim_i(R_{A_i})$. Thus, if each $R_{A_i}$ is an $\ce$-quotient of some $Y_i\in\cv_\lambda$, then $\colim_i(R_{A_i})$ is an $\ce$-quotient of $Y:=\sum_iY_i$, which is still in $\cv_\lambda$. It follows that $R_A$ is also an $\ce$-quotient of $Y$. Thus the objects satisfying the second condition are clearly closed under $\lambda$-small colimits. \\
	Since $0$ is stable under $\lambda$-small colimits, then also are its $\ce$-quotients (by the arguments above); therefore also the objects satisfying the third condition are closed under $\lambda$-small colimits.
	
	(3). Notice that, since the projections into $\Str(\FF)$ and into $\cm$ (for any relation in $\mathbb L$) are jointly conservative, to show that $\cg$ is a strong generator of $\Str(\LL)_0$ it is enough to show the following for any $\mathbb L$-structure $L$. (a) That any morphism of $\mathbb F$-structures $B\to J_\mathbb F(L)$, with $B$ $\lambda$-presentable, factors through some $J_\mathbb F(A\to L)$ with $A\in\cg$. (b) That, for any relation $R:X$, any morphism $m\to r_L$ in $\cm$, with $m\in\cm_\lambda$, factors through the $R$-component of some $A\to L$ with $A\in\cg$. \\
	Point (a) is trivial since, by definition $J_\mathbb F(\cg)=  \Str(\FF)_\lambda$. For (b), consider a relation $R:X$ and any morphism $(u,v)\colon m\to r_L$ in $\cm$, with $m\colon B\rightarrowtail C$ in $\cm_\lambda$; so that $u\colon B\to R_L$, $v\colon C\to L^X$, and $r_Lu=vm$. Note that $m\colon B\rightarrowtail C$ is $\lambda$-presentable in $\cm$ if and only if $C\in\cv_\lambda$ and $B$ is an $\ce$-quotient of some $Y\in\cv_\lambda$.\\ 
	Since we can write the $\mathbb F$-structure $J_\mathbb F(L)$ as a $\lambda$-filtered colimit of $\lambda$-presentable $\mathbb F$-structures, and this colimit is computed as in $\cv$, the map $v$ factors through some $h^X\colon A^X\to L^X$ arising from an $h\colon A\to L$ in $\Str(\FF)$. Therefore we have the solid part of the diagram below
	\begin{center}
		\begin{tikzpicture}[baseline=(current  bounding  box.south), scale=2]
			
			\node (0) at (-2,0) {$Y$};
			\node (a) at (-1,-0.8) {$C$};
			\node (a0) at (0,-0.8) {$A^X$};
			\node (b0) at (1,-0.8) {$L^X$};
			\node (b) at (-1,0) {$B$};
			\node (c0) at (0,0) {$R_A$};
			\node (d0) at (1,0) {$R_L$};
			
			\path[font=\scriptsize]
			
			(0) edge [->>] node [above] {$\ce$} (b)
			(a) edge [->] node [below] {} (a0)
			(a) edge [<-<] node [left] {$m$} (b)
			(b) edge [dashed, ->>] node [below] {$\ce$} (c0)
			
			(a0) edge [->] node [below] {$h^X$} (b0)
			(a0) edge [dashed, <-<] node [left] {$\cm$} (c0)
			(b0) edge [<-<] node [right] {$r_L$} (d0)
			(c0) edge [dashed, ->] node [below] {$h_R$} (d0)
			
			(b) edge [bend left,->] node [above] {$u$} (d0)
			(a) edge [bend right,->] node [below] {$v$} (b0);
		\end{tikzpicture}	
	\end{center}
	and we define $R_A$ to be the $(\ce,\cm)$-factorization of the composite $B\to A^X$ and the map $h_R\colon R_A\to R_L$ is induced by orthogonality of $\ce$ and $\cm$. It follows that we can extend $A$ to an $\mathbb L$-structure in $\cg$ by defining the interpretation of $R:X$ as $R_A$ (which is an $\ce$-quotient of $Y\in\cv_\lambda$) and setting all the other relations to be $\ce$-quotients of $0$. By construction $h$ extends to a morphism of $\mathbb L$-structures which satisfies our condition.
\end{proof}

\section{Formulas}

Given a language $\LL$, we will introduce enriched atomic formulas over $\LL$. Starting from these, we take conjunctions and existential quantifications and define the regular fragment of our theory. Then we define their interpretation into any $\LL$-structure using the given factorization system on $\cv$.

The assumptions for this section are the same as those in the previous one. Given a language $\mathbb L=\mathbb F\cup \mathbb R$ the terms we consider are the $(X,Y)$-ary extended $\mathbb F$-terms defined in \cite[Section~4]{RTe}; these include (but may not restrict to) the recursively generated ones of \cite[Definition~4.1]{RTe}.

\begin{nota}
Variables have arities which are objects of $\cv$ and we denote them as $x:X$, for $X\in\cv$; as extended terms they correspond to the identity maps between arities. Then, given an $(X,Z)$-ary term $t$ we write $t(x)$ to emphasize that $t$ has input arity $X$. If $t$ has input arity $X+Y$ we write $t(x,y)$ to specify that we have two input variables; this helps notationally for existential quantification.
\end{nota}
 
As mentioned before, for the purposes of this paper we shall not go further the regular fragment of our logic.

\begin{defi}
	Given a language $\LL$, the atomic formulas of $\LL$ are defined as follows:
	\begin{enumerate}
		\item if $s,t$ are $(X,Y)$-ary terms, then 
		$$\varphi(x):=(s(x)=t(x))$$ is an $X$-ary atomic formula;
		\item if $R$ is a $X$-ary relation symbol, $Y$ is an arity, and $t$ an $(Z,X\otimes Y)$-ary term, then 
		$$\varphi(z):= R^Y(t(z))$$ is a $Z$-ary atomic formula.
	\end{enumerate}
	General formulas are built recursively from atomic formulas by taking:
		\begin{enumerate}
		\item[(3)] {\em conjunctions}: if $\varphi_j(x)$ are $X$-ary formulas, then  
		$$\varphi(x):=\bigwedge\limits_{j\in J}\varphi_j$$ is an $X$-ary formula;
		\item[(4)] {\em existential quantification}: if $\psi(x,y)$ is an $X+Y$-ary formula, then
		$$ \varphi(x):= (\exists y)\psi(x,y) $$
		is an $X$-ary formula.
	\end{enumerate}
	
	If $\LL$ is $\lambda$-ary, the arities involved are $\lambda$-presentable, and $|J|<\kappa$, we get $\mathbb L_{\kappa\lambda}$-formulas. 
\end{defi}

\begin{defi}
	A {\em positive-primitive formula}, also called {\em pp-formula}, is one of the form
	$$ \varphi(x)\equiv (\exists y) \psi(x,y) $$
	where $\psi$ is a conjunction of atomic formulas.
\end{defi}

Next we define the interpretation of formulas in $\mathbb L$-structures. The interpretation of terms is defined in \cite{RTe}: an $(X,Y)$-ary $t$ term interprets on an $\mathbb L$-structure $A$ as a morphism $t_A\colon A^X\to A^Y$.

\begin{rem}\label{constants}
If $\cv$ is semi-cartesian (that is, the unit $I$ is terminal) we obtain $Y$-ary constant symbols as function symbols of arity $(0,Y)$. 
\end{rem}

\begin{defi}
	For any $\LL$-structure $A$ and any $X$-ary formula $\varphi(x)$ we define its interpretation in $A$ as an $\cm$-subobject
	$$ \varphi_A\rightarrowtail A^X.$$
	We argue recursively as follows:
	{\setlength{\leftmargini}{1.6em}
		\begin{enumerate}
	\item If $\varphi(x)\equiv(s(x)=t(x))$, then the interpretation $\varphi_A\rightarrowtail A^X$ is defined as the $(\ce,\cm)$-factorization of the 	the equalizer of $s_A$ and $t_A$:
	\begin{center}
		\begin{tikzpicture}[baseline=(current  bounding  box.south), scale=2]
			
			\node (a0) at (0,0) {$A_{s,t}$};
			\node (c0) at (1.8,0) {$ A^X $};
			\node (d0) at (0.9,0.5) {$\varphi_A$};
			
			\path[font=\scriptsize]
			
			(a0) edge [->] node [below] {$\tx{eq}(s_A,t_A)$} (c0)
			(a0) edge [->>] node [above] {$\ce\ \ \ $} (d0)
			(d0) edge [>->] node [above] {$\ \ \ \cm$} (c0);
		\end{tikzpicture}	
	\end{center}
	\item If $\varphi(z)\equiv R^Y(t(z))$, then the interpretation $\varphi_A\rightarrowtail A^Z$ is defined by the pullback
	\begin{center}
		\begin{tikzpicture}[baseline=(current  bounding  box.south), scale=2]
			
			\node (a0) at (0,0.8) {$\varphi_A$};
			\node (a0') at (0.2,0.6) {$\lrcorner$};
			\node (b0) at (1,0.8) {$A^Z$};
			\node (c0) at (0,0) {$R_A^Y$};
			\node (d0) at (1,0) {$A^{X\otimes Y}$};
			
			\path[font=\scriptsize]
			
			(a0) edge [>->] node [above] {} (b0)
			(a0) edge [->] node [left] {} (c0)
			(b0) edge [->] node [right] {$t_A$} (d0)
			(c0) edge [>->] node [below] {$r_A^Y$} (d0);
		\end{tikzpicture}	
	\end{center}
	where we composed $r_A^Y$ with the isomorphism $(A^X)^Y\cong A^{X\otimes Y}$  (note, the top arrow is still in $\cm$ since it is closed under powers and pullbacks in $\cv^\to$).
	\item If $ \varphi(x)\equiv\bigwedge_{j\in J}\varphi_j$, then $\varphi_A\rightarrowtail A^X$ is the intersection (wide pullback) of all the $(\varphi_j)_A\rightarrowtail A^X$ for $j\in J$. (Since $\cm$ is closed under intersections, this map is in $\cm$).
	\item If $\varphi(x)\equiv((\exists y)\psi(x,y))$, then $\varphi_A\rightarrowtail A^X$ is given by the $(\ce,\cm)$-factorization below
	\begin{center}
		\begin{tikzpicture}[baseline=(current  bounding  box.south), scale=2]
			
			\node (a0) at (-0.1,0.8) {$\psi_A$};
			\node (c0) at (1,0.8) {$ A^{X+Y} $};
			\node (c1) at (2.1,0.8) {$ A^X$};
			\node (d0) at (1,1.3) {$\varphi_A$};
			
			\path[font=\scriptsize]
			
			(a0) edge [>->] node [below] {} (c0)
			(c0) edge [->] node [below] {$p_1$} (c1)
			(a0) edge [->>] node [above] {$\ce\ \ \ $} (d0)
			(d0) edge [>->] node [above] {$\ \ \ \cm$} (c1);
		\end{tikzpicture}	
	\end{center}
	where the bottom composite is given by the $\cm$-morphism defining $\psi_A$ and the projection on the first factor.
	\end{enumerate}}
\end{defi}	

\begin{rem}\label{M-regularmon}
	If $(\ce,\cm)$ is proper then $\cm$ contains the regular monomorphisms (\cite[Proposition~2.1.4(d)]{FK}) so the interpretation of $(s(x)=t(x))$ in $A$ coincides with the equalizer of $s_A$ and $t_A$.
\end{rem}

\begin{nota}\label{equation}
	If $s$ and $t$ are respectively $(X,Z)$-ary and $(Y,Z)$-ary terms, then we denote by $(s(x)=t(y))$
	the $(X+Y)$-ary atomic formula defined by
	$$(s(x)=t(y)):=(\ s(i_X(x,y))=t( i_Y(x,y))\ )$$
	where $i_X$ and $i_Y$ are respectively the inclusions of $X$ and $Y$ in $X+Y$. It is easy to see that the interpretation $\varphi_A\rightarrowtail A^{X+Y}\cong A^X\times A^Y$ is defined by the $(\ce,\cm)$-factorization of the map
	\begin{center}
		\begin{tikzpicture}[baseline=(current  bounding  box.south), scale=2]
			
			\node (a0) at (0,0.8) {$\hat A_{s,t}$};
			\node (c0) at (1.3,0.8) {$ A^X\times A^Z $};
			
			\path[font=\scriptsize]
			
			(a0) edge [>->] node [below] {$h$} (c0);
		\end{tikzpicture}	
	\end{center}
	induced by the pullback of $s_A$ along $t_A$ (meaning that $h$ is the morphism induced into the product).
\end{nota}

\begin{lemma}\label{product} The following properties hold:
{\setlength{\leftmargini}{1.6em}
	\begin{enumerate}
		\item If $\ce$ is closed under products in $\cv^\to$, then 	for every positive-primitive formula $\varphi$
		$$
		\varphi_{\prod A_i}\cong\prod\varphi_{A_i}.
		$$
		\item If $X\in\cv$ is an $\ce$-stable object, then for every positive-primitive formula $\varphi$ 
		$$
		\varphi_{A^X}\cong(\varphi_{A})^X.
		$$
		\item If $A\cong\colim A_i$ is a $\lambda$-directed colimit of $\LL$-structures, then
		$$
		\varphi_A\cong\colim\varphi_{A_i}
		$$
		for every $\lambda$-ary positive-primitive formula $\varphi$.
	\end{enumerate}}
\end{lemma}
\begin{proof}
	These follow directly from how we defined interpretation starting from atomic formulas, using that $(\ce,\cm)$-factorizations are stable under the limits and colimits considered in the statement (given the additional hypotheses).
\end{proof}

Now we turn to the notion of satisfaction.

	\begin{defi}
		Given an $\mathbb L$-structure $A$, an $X$-ary formula $\varphi(x)$, and an arrow $a\colon X\to A$ in $\cv$ (a generalized element of $A$), we say that {\em $A$ satisfies $\varphi[a]$} and write
		$$ A\models\varphi[a]  $$
		if the transpose $\hat a\colon I\to A^X$ of $a$ factors through $\varphi_A\rightarrowtail A^X$.
	\end{defi}

Then we have two possible approaches to satisfaction:
	\begin{defi}
		Given an $X$-ary formula $\varphi(x)$, we say that $A\in\Str(\LL)$ {\em satisfies $\varphi(x)$}, and write
		$$A\models \varphi,$$ 
		if $\varphi_A\rightarrowtail A^X$ is an isomorphism. We say that {\em $A$ satisfies $\varphi(x)$ pointwise} if $ A\models\varphi[a] $ for any $a\colon X\to A$.
	\end{defi}

Clearly, satisfaction implies pointwise satisfaction. For the converse,
we need that $\ce$ contains the class of surjections. Indeed, if $A$ satisfies $\varphi(x)$ pointwise, then the $\cm$-subobject $\varphi_A\rightarrowtail A^X$ is (a surjection and thus) in $\ce$, so that $\varphi_A\cong A^X$. 

We extend the satisfaction from formulas to \textit{sequents}.

\begin{defi}
Given $X$-ary formulas $\varphi$ and $\psi$, we say that $A\in\Str(\LL)$ {\em satisfies the sequent} $
(\forall x)(  \varphi(x) \vdash \psi(x)),
$ and write
		$$ A\models( \varphi\vdash \psi), $$
		if $\varphi_A\subseteq \psi_A$ as $\cm$-subobjects of $A^X$. We say that $A$ {\em satisfies the sequent pointwise} if for any $a$ such that $A\models\varphi[a]$, then 
$A\models\psi[a]$.
	\end{defi}

\begin{defi}
	Given a set $\TT$ of sequents, we denote by $\Mod(\TT)$ the full subcategory of $\Str(\LL)$ spanned by those $\LL$-structures that satisfy the sequents in $\TT$.
\end{defi}

	\begin{rem}
		Note that the sequent $(\forall x)(  \varphi(x) \vdash \psi(x))$ is not itself a formula in our sense. Moreover, an $\LL$-structure $A$ is such that $A\models \varphi$ if and only if it satisfies the sequent $\top\vdash \varphi$, where $\top$ is the empty conjunction.
	\end{rem}

	Again, satisfaction implies pointwise satisfaction. While, if $\ce$ contains the surjections, the two notions are equivalent: consider the pullback below;
	\begin{center}
		\begin{tikzpicture}[baseline=(current  bounding  box.south), scale=2]
			
			\node (a0) at (0,0.8) {$(\varphi\wedge\psi)_A$};
			\node (a0') at (0.2,0.6) {$\lrcorner$};
			\node (b0) at (1.2,0.8) {$\varphi_A$};
			\node (c0) at (0,0) {$\psi_A$};
			\node (d0) at (1.2,0) {$A^X$};
			
			\path[font=\scriptsize]
			
			(a0) edge [>->] node [above] {$f$} (b0)
			(a0) edge [>->] node [left] {} (c0)
			(b0) edge [>->] node [right] {} (d0)
			(c0) edge [>->] node [below] {} (d0);
		\end{tikzpicture}	
	\end{center}
if $A$ satisfies the sequent pointwise, then $f$ is a surjection and hence in $\ce$, but it is also in $\cm$; thus $f$ is an isomorphism and $\varphi_A\subseteq \psi_A$.

The following remark shows that not all deduction rules of ordinary regular logic apply in this context; the reason is that the class of maps $\ce$ is not stable under pullbacks in general.

\begin{rem}\label{existential-wedge}
	Not all rules of predicate logic still apply in this enriched setting. For instance, the {\em Frobenius condition} stating that the two formulas
	$$ (\exists y)(\varphi(x)\wedge\psi(x,y))\ \ \ \ \ \  \text{ and }\ \ \ \ \ \ \varphi(x)\wedge(\exists y)\psi(x,y) $$
	are equivalent, might not hold in our setting.
	Indeed, given an $X$-ary formula $\varphi$ and an $X+Y$-ary formula $\psi$, the interpretation of $(\varphi(x)\wedge\psi(x,y))$ in an $\LL$-structure $A$ is given by the pullback:
	\begin{center}
		\begin{tikzpicture}[baseline=(current  bounding  box.south), scale=2]
			
			\node (a0) at (0,0.8) {$(\varphi\wedge\psi)_A$};
			\node (a0') at (0.2,0.6) {$\lrcorner$};
			\node (b0) at (2.4,0.8) {$\varphi_A$};
			\node (c0) at (0,0) {$\psi_A$};
			\node (d0) at (1.2,0) {$A^{X+Y}$};
			\node (e0) at (2.4,0) {$A^X$};
			
			\path[font=\scriptsize]
			
			(a0) edge [->] node [above] {} (b0)
			(a0) edge [>->] node [left] {} (c0)
			(b0) edge [>->] node [right] {} (e0)
			(c0) edge [>->] node [below] {} (d0)
			(d0) edge [->] node [below] {} (e0);
		\end{tikzpicture}	
	\end{center}
	Now, there are two ways in which one can take existential quantification over $y:Y$; we could consider either
	$$ \varphi(x)\wedge(\exists y)\psi(x,y) $$
	or 
	$$ (\exists y)(\varphi(x)\wedge\psi(x,y)). $$
	Their interpretation in an $\LL$-structure $A$ is given as follows
	\begin{center}
		\begin{tikzpicture}[baseline=(current  bounding  box.south), scale=2]
			
			\node (a0) at (0,0.8) {$(\varphi\wedge\psi)_A$};
			\node (a0') at (0.2,0.6) {$\lrcorner$};
			\node (c0) at (0,0) {$\psi_A$};
			
			\node (d0) at (1.4,0) {$(\exists y \psi)_A$};
			\node (d0') at (1.6,0.6) {$\lrcorner$};
			\node (e'0) at (1.4,0.8) {$(\varphi\wedge\exists y \psi)_A$};
			\node (e''0) at (1.4,1.4) {$(\exists y (\varphi\wedge\psi))_A$};
			
			\node (b0) at (2.8,0.8) {$\varphi_A$};
			\node (e0) at (2.8,0) {$A^X$};
			
			\path[font=\scriptsize]
			
			(a0) edge [->] node [above] {} (e'0)
			(e'0) edge [>->] node [above] {} (b0)
			
			(a0) edge [bend left=20, ->>] node [above] {} (e''0)
			(e''0) edge [bend left=20,>->] node [above] {} (b0)
			(e''0) edge [dashed, >->] node [below] {} (e'0)
			
			(a0) edge [>->] node [left] {} (c0)
			(b0) edge [>->] node [right] {} (e0)
			(e'0) edge [>->] node [right] {} (d0)
			
			(c0) edge [->>] node [below] {} (d0)
			(d0) edge [>->] node [below] {} (e0);
		\end{tikzpicture}	
	\end{center}
	where $(\exists y \psi)_A$ is the $(\ce,\cm)$-factorization of the bottom row of the previous diagram, and $(\exists y (\varphi\wedge\psi))_A$ is the $(\ce,\cm)$-factorization of the top row followed by the vertical right one (which coincides with the $(\ce,\cm)$-factorization of the top row). The dashed arrow above exists by orthogonality of the factorization system.
	
	It follows that in general any $\LL$-structure $A$ satisfies the sequent
	$$  (\forall x)\ (\exists y)(\varphi(x)\wedge\psi(x,y))\ \vdash\ \varphi(x)\wedge(\exists y)\psi(x,y), $$
	but need not satisfy the other implication (see below for a concrete counterexample). This will result in a non-standard choice of the notion of regular theories in Section~\ref{regular-section}.
	However, it is easy to see that the inverse sequent 
	$$  (\forall x)\ \varphi(x)\wedge(\exists y)\psi(x,y)\ \vdash\ (\exists y)(\varphi(x)\wedge\psi(x,y)),$$
	 and hence the full Frobenius condition, holds with the additional assumption that $\ce$ is stable under pullbacks.

\end{rem}

\begin{exam}
	Let $\cv=\Met$ with the (dense, closed isometry) factorization system, where dense maps are known to not be stable under pullbacks. Consider the language $\LL$ with two relation symbols $R:1$ and $S:1+1$, so that an $\LL$-structure is a metric space $M$ together with two closed subspaces $R_M\subseteq M$ and $S_M\subseteq M\times M$. Within the notations of the previous remark we consider the formulas $\varphi(x)=R(x)$ and $\psi(x,y)=S(x,y)$.
	
	Fix then the $\LL$-structure 
	$$M:=(\mathbb R,R_M:=\{0\}, S_M:=\{(1/n,n)\}_{n>0});$$
	it is easy to see that
	$$ (\exists y (R(x)\wedge S(x,y))_M = \emptyset$$
	while
	$$ (R(x)\wedge\exists y S(x,y))_M=\{0\}\cap (\{1/n\}_{n>0}\cup\{0\})=\{0\}$$
	so that the interpretation of the two formulas is not the same.
\end{exam}

In the next example we show how to (partially) capture Iovino's languages over Banach spaces~\cite{I} into our context.

\begin{exam}\label{Iovino}
	Fix $\cv$ to be the category $\Ban$ of Banach spaces. Every enriched language $\mathbb L$ over $\Ban$ contains a $(\mathbb C^2,\mathbb C)$-ary term $+$ induced by the diagonal $\Delta:\mathbb C\to\mathbb C^2$
	(see \cite[4.1]{RTe}), as well as $(\mathbb C,\mathbb C)$-ary terms $c\cdot-$ for $c\in\mathbb C$ with $|c|\leq 1$ induced by the morphisms $c\cdot-:\mathbb C\to\mathbb C$, and the $(0,\mathbb C)$-ary term induced by $\mathbb C\to 0$. 
	
	For any $t>0$, denote by $\mathbb C_t$ the Banach space given by $\mathbb C$ with norm $t\cdot|-|\colon\mathbb C\to\mathbb R$. Then, given $B\in\Ban$ the space $B^{\mathbb C_t}$ is the same as $B$ equipped with the norm $t^{-1}\cdot \parallel\!-\!\parallel_B$. It follows that to give a morphism $B=B^\mathbb C\to B^{\mathbb C_t}$ in $\Ban$ is the same as giving an operator $T\colon B\to B$ of norm $\parallel \! T\!\parallel\ \leq t$.
	
	The languages from \cite{I} consists of sets of symbols for operators and symbols for constants, each equipped with an upper norm bound. In our setting we can see the operators as $(\mathbb C,\mathbb C_t)$-ary function symbols 
	here, $t$ is the upper bound for the norm. 
	
	Positive bounded formulas of \cite{I} include $\parallel\! x\!\parallel\ \leq r$ and $\parallel\! x\!\parallel\ \geq r$ where $r$ is a positive rational number. In our setting, we automatically have $\parallel\! x\!\parallel\ \leq r$ (by taking $x$ to be $\mathbb C$-ary) so that $\parallel\! x\!\parallel\ \leq r$ can be expressed as $\parallel\!\frac{1}{r}x\!\parallel\ \leq 1$ for $r\geq 1$. For $r<1$,
we get $\parallel\! x\!\parallel\ \leq r$ by taking $x$ to be $\mathbb C_r$-ary. Formulas $\parallel\! x\!\parallel\ \geq r$, as well as constant symbols, seem to be beyond our setting. 
	
	For positive bounded formulas, the satisfaction $\models$ from \cite{I} is our satisfaction with respect to the factorization system (strong epimorphisms, monomorphisms) while the approximate satisfaction $\models_\ca$ is our satisfaction with respect to (epimorphisms, strong monomorphisms). The fact that $\models$ is stronger than $\models_\ca$ corresponds to our \ref{depend}.
\end{exam}

\begin{nota}\label{unique}
Let $\psi(x,y)$ be an $X+Y$-ary formula. We say that an $\mathbb L$-structure $A$ satisfies $(\exists !y)\psi(x,y$) if it satisfies the formula $(\exists y)\psi(x,y)$ and the sequent
$$
(\forall x)(\forall y)(\forall y')(\psi(x,y)\wedge\psi(x,y')\ \vdash\ y=y').
$$
When considering a sequent of the form
$$
(\forall x)(  \varphi(x) \vdash (\exists !y)\psi(x,y)),
$$
we say that $A$ satisfies it if it satisfies
$$
(\forall x)(  \varphi(x) \vdash (\exists y)\psi(x,y)),
$$
and
$$
(\forall x)(\forall y)(\forall y')(  \varphi(x) \wedge \psi(x,y)\wedge\psi(x,y')\ \vdash \ y=y').
$$

\end{nota}

\begin{exam}\label{limits}
	Following Example~\ref{Sketches}, we consider $\cv=\bo{Cat}$ with the (surjective on objects, injective on objects fully faithful) factorization system. Given $C\in\bo{Cat}_\lambda$ we define a language $\mathbb L$ and formulas expressing the existence of $C$-limits in a category $A$. Since surjective on objects = surjections, satisfaction and pointwise satisfaction coincide.
	
	Similarly to~\ref{Sketches}, given $C\in\Cat_\lambda$ we define $2*C$ to be the category obtained by adding to $0*C$ a morphism $1\to 0$ with codomain the initial object of $0*C$. We have two inclusions: $j_0,j_1\colon 0*C\to 2*C$ sending $C$ to itself and picking out $0$ and $1$ respectively. Finally, define $i*C$ as $2*C$ where the added morphism is invertible, this also has two inclusions $l_0,l_1\colon 0*C\to i*C$.
	
	Now, Consider the language $\mathbb L$ with just a relation symbol $R:0*C$; for each $\mathbb L$-structure $A$, $R_A$ should be thought of as the set of limiting cones over $C$. First, let us define the formula $\psi(x,z):(0*C)+(0*C)$ as below
	$$
	\psi(x,z)\equiv \exists! (w:2*C)\ ( j_0(w)=x \wedge j_1(w)=z)
	$$
	stating the existence of a unique factorization of $z$ through $x$. Now we define sequents
	$$ 
	\alpha\equiv (\forall x:0*C)(\forall z:0*C)\  (R(x) \wedge (k(x)=k(z))) \vdash \psi(x,z),$$
	following the notation introduced above, then
	$$ \beta\equiv (\forall y:C) (\exists x:0*C)\ R(x)\wedge (k(x)=y), $$
	and
	$$ \gamma\equiv (\forall w:i*C)\ R(l_0(w))\vdash R(l_1(w)). $$
	Then $\alpha$ says that every element of $R$ is a limiting cone, $\beta$ says that every diagram from $C$ has a limiting cone in $R$, and $\gamma$ says that $R$ contains all limiting cones. Therefore, to give an element of $\Mod(\alpha,\beta,\gamma)$ is the same as to give a category $A$ where every $C\to A$ has a limit in $A$; then $R_A$ is univocally determined as the full subcategory of $A^{0*C}$ spanned by the limiting cones. A morphism of models needs to preserve such full subcategories; thus sends limiting cones to limiting cones. It follows that $\Mod(\alpha,\beta,\gamma)$ is the 2-category of small categories with $C$-limits, $C$-limit preserving functors, and natural transformations. 
\end{exam}

\begin{exam}
Let $\cv=\bo{Cat}$ and consider the functional language $\mathbb L_0$ with one $(1,2)$-ary function symbol $\rho$ where $2=\{0\to 1\}$ is the arrow category. Denote $T:=i_0\circ\rho$ and consider the equation $i_1\circ\rho=\id$ where $i_0,i_1:1\to 2$ are inclusion. Models 
of this theory $E_0$ are categories $A$ equipped with a functor $T\colon A\to A$
and a natural transformation $\rho\colon T\to \id_A$ (see \cite[Example~5.10]{RTe}).

Add $(1,2)$-ary function symbols $+,\pi_1,\pi_2$, denote $T_2:=i_0\circ +$ and add the equations $i_1\circ +=i_1\circ \pi_1=i_1\circ \pi_1=T$ and $i_0\circ \pi_1=i_0\circ \pi_2=T_2$. This yields natural transformations
$+,\pi_1,\pi_2:T_2\to T$. Using Example~\ref{limits} above, we can arrange that
\begin{center}
		\begin{tikzpicture}[baseline=(current  bounding  box.south), scale=2]
			
			\node (a0) at (0,0.8) {$T_2$};
			\node (b0) at (1,0.8) {$T$};
			\node (c0) at (0,0) {$T$};
			\node (d0) at (1,0) {$\id$};
			
			\path[font=\scriptsize]
			
			(a0) edge [->] node [above] {$\pi_1$} (b0)
			(a0) edge [->] node [left] {$\pi_2$} (c0)
			(b0) edge [->] node [right] {$\rho$} (d0)
			(c0) edge [->] node [below] {$\rho$} (d0);
		\end{tikzpicture}	
	\end{center}
is a pullback. In detail, let $C$ be the category
\begin{center}
		\begin{tikzpicture}[baseline=(current  bounding  box.south), scale=2]
			
			\node (b0) at (1,0.8) {$\cdot$};
			\node (c0) at (0,0) {$\cdot$};
			\node (d0) at (1,0) {$\cdot$};
			
			\path[font=\scriptsize]
			
			(b0) edge [->] node [right] {} (d0)
			(c0) edge [->] node [below] {} (d0);
		\end{tikzpicture}	
	\end{center}
giving the shape for pullbacks, and $\psi(x,z)$ be the formula of arity $(0*C)+(0*C)$  from Example~\ref{limits}; note that we have a term $t:=(\pi_1,\pi_2,\rho,\rho)$ of arity $(1,0*C)$ which picks out the cone of shape $C$ that we want express as a limit. Given a structure $A$, then $T_2$ is the pullback of $\rho$ along itself if and only if and only if for each $a\in A$ the cone $t(a):0*C\to A$ is a limit cone in $A$, if and only if 
$$ A\models \psi(tx,z),$$
where now $\psi(tx,z)$ has arity $1+(0*C)$.

Using the category
\begin{center}
		\begin{tikzpicture}[baseline=(current  bounding  box.south), scale=2]
			
			\node (a0) at (0.34,0.6) {$\cdot$};
			\node (b0) at (1,0.8) {$\cdot$};
			\node (c0) at (0,0) {$\cdot$};
			\node (d0) at (1,0) {$\cdot$};
			
			\path[font=\scriptsize]
			
			(a0) edge [->] node [right] {} (d0)
			(b0) edge [->] node [right] {} (d0)
			(c0) edge [->] node [below] {} (d0);
		\end{tikzpicture}	
	\end{center}
we express that $+$ is associative.	This makes $T$ a semigroup bundle; that is, a semigroup in the category of functors over $\id$. Hence models
of the resulting theory $E$ are categories equipped with a semigroup bundle over $\id$. Morphisms are functors strictly preserving the semigroup bundles. 

Similarly, we can describe categories equipped with a commutative group bundle (see \cite{R4}) and categories equipped with additive bundle (= commutative monoid bundle, \cite{CC}). In this way, we can get  tangent categories (see \cite{R4,CC}) where morphisms strictly preserve
tangent structures. But right morphisms $F$ of tangent categories seem to preserve tangent structure up to an isomorphism, i.e. $FT_1\cong T_2F$ (see \cite{CC}).
\end{exam}

\begin{rem}\label{depend}
For a functional language, $\mathbb L$-structures do not depend on the chosen factorization system $(\ce,\cm)$. However, for a general language they do. Assume that we have factorization systems $(\ce_0,\cm_0)$ and $(\ce_1,\cm_1)$ such that $\ce_0\subseteq\ce_1$, hence $\cm_1\subseteq\cm_0$. Then every $\mathbb L$-structure with respect to $\cm_1$ is an $\mathbb L$-structure with respect to $\cm_0$, giving a fully faithful inclusion
$$ \Str(\mathbb L)^1\hookrightarrow\Str(\mathbb L)^0,$$
where $1$ and $0$ denote the factorization systems used to define the $\LL$-structures. Note that this reflects (but in general does not preserve) the $\lambda$-presentable objects by~\ref{lambda-pres}.

For a conjunction $\varphi(x)$ of atomic formulas and an $\mathbb L$-structure $A$ with respect to $\cm_1$ the satisfaction $A\models\varphi$ 
does not depend on the factorization system. However, if $\varphi(x)$ is positive-primitive then satisfaction $A\models\varphi$ with respect to $\ce_0$ is stronger than that with respect to $\ce_1$.

\end{rem}

\section{Presentation formulas}\label{sect:pres}

From this section our factorization system is assumed to be proper.

\begin{assume}\label{assum}
	We fix a proper enriched factorization system $(\ce,\cm)$ on $\cv$ which is closed in $\cv^\to$ under $\lambda$-filtered colimits.
\end{assume}

In particular $\cm$ contains the regular monomorphisms, and hence the interpretation of equations does not involve taking $(\ce,\cm)$-factorizations (see Remark~\ref{M-regularmon}). This will be relevant below.

With this definition we generalize the notion of presentation formula form \cite{AR}.

\begin{defi}
	Given a $\lambda$-presentable $\mathbb L$-structure $A$, we say that an $X$-ary formula $\pi^A(x)$ of $\mathbb L_{\lambda\lambda}$ is a \textit{presentation formula} of $A$ if for every $\mathbb L$-structure $B$ we have a bijection, natural in $B$, between elements $a\colon X\to B$ for which $B\models\pi^A[a]$ and morphisms $\bar a\colon A\to B$ of $\mathbb L$-structures. 
\end{defi}

We shall now give a more categorical interpretation of this definition. Every formula $\varphi$ that is a conjunction of atomic formulas, induces a $\cv$-functor
$$ \varphi_{(-)}\colon \Str(\LL)\longrightarrow \cv$$ 
defined by sending $A$ to the interpretation $\varphi_A$ of $\varphi$. This can be constructed by taking certain limits of the forgetful $\cv$-functors $U\colon \Str(\LL)\to \cv$ and $V\colon \Str(\LL)\to \prod_\mathbb R\cm$; the limits are those involved in the definition of interpretation (which does not involve $(\ce,\cm)$ factorizations). It follows in particular that $\varphi_{(-)}$ is always continuous.

\begin{propo}
	A conjunction of atomic formulas $\varphi(x)$ is a presentation formula for $A$ if and only if $$\Str(\LL)(A,-)\cong \varphi_{(-)}.$$
\end{propo}
\begin{proof}
	By definition of satisfaction, $\varphi$ is a presentation for $A$ if and only if
	$$ \Str(\LL)_0(A,-)\cong \cv_0(I,(\varphi_{(-)})_0). $$
	But $\Str(\LL)_0(A,-)$ is by definition $\cv_0(I,\Str(\LL)(A,-)_0)$, and both $\Str(\LL)(A,-)$ and $\varphi_{(-)}$ preserve powers (being continuous); thus the existence of the natural isomorphism above is equivalent to having
	$$ \cv_0(X,\Str(\LL)(A,-)_0)\cong \cv_0(X,(\varphi_{(-)})_0) $$
	naturally in $X\in\cv_0$. And this is in turn equivalent to $\Str(\LL)(A,-)\cong \varphi_{(-)}$.
\end{proof}

\begin{coro}\label{conj->pres}
	Let $\varphi$ be a $\lambda$-ary conjunction of atomic formulas in a language $\mathbb L$. Then $\varphi$ is a presentation formula of some $\lambda$-presentable object of $\Str(\LL)$.
\end{coro}
\begin{proof}
	The $\cv$-functor $ \varphi_{(-)}\colon \Str(\LL)\longrightarrow \cv$ is continuous and preserves $\lambda$-filtered colimits since it is defined by taking $\lambda$-small limits of the forgetful $\cv$-functors $U$ and $V$. Thus $\varphi_{(-)}$ has a left adjoint, and is therefore a representable $\cv$-functor. It follows that the representing object $A$ is $\lambda$-presentable and, by the proposition above, $\varphi$ is a presentation formula for $A$. 
\end{proof}

\begin{lemma}\label{strong-pres}
Let $\mathbb L=\mathbb F\cup\mathbb R$ be a language whose functional part $\mathbb F$ has $\ce$-stable input arities. Then:\begin{enumerate}
		\item for any $g\colon A\to B$ in $\Str(\mathbb L)$, the $(\ce,\cm)$ factorization of the underlying morphism $Ug$ in $\cv$ lifts to a factorization $(e,m)$ of $g$ in $\Str(\mathbb L)$;
		\item $U\colon\Str(\LL)\to\cv$ sends regular epimorphisms to maps in $\ce$. 
	\end{enumerate}
\end{lemma} 
\begin{proof}
	$(1)$. Consider a morphism $g\colon A\to B$ in $\Str(\mathbb L)$, and the $(\ce,\cm)$ factorization $(e\colon A\to E, m\colon E\to B)$  of $Ug$ in $\cv$. By \cite[Lemma~B.5.(1)]{RTe} we know that $E$ inherits a unique notion of $\mathbb F$-structure making $e$ and $m$ morphisms in $\Str( \mathbb F)$; thus we are left to prove that $E$ also inherits a compatible notion of $\mathbb R$-structure.
	
	For any $X$-ary relation symbol $R$ in $\mathbb L$, we can consider the solid part of the diagram below in $\cv$.
	\begin{center}
		\begin{tikzpicture}[baseline=(current  bounding  box.south), scale=2]
			
			\node (a) at (-1,-0.8) {$A^X$};
			\node (a0) at (0,-0.8) {$E^X$};
			\node (b0) at (1,-0.8) {$B^X$};
			
			\node (b) at (-1,0) {$R_A$};
			\node (c0) at (0,0) {$R_E$};
			\node (d0) at (1,0) {$R_B$};
			
			\path[font=\scriptsize]
			
			(a) edge [->] node [above] {$e^X$} (a0)
			(b) edge [dashed, ->>] node [above] {$e_R$} (c0)
			
			(a0) edge [>->] node [above] {$m^Y$} (b0)
			(c0) edge [dashed, >->] node [above] {$m_R$} (d0)
			
			(a) edge [<-<] node [left] {$r_A$} (b)
			(a0) edge [dashed, <-<] node [left] {$r_E$} (c0)
			(b0) edge [<-<] node [right] {$r_B$} (d0)
			
			(b) edge [bend left,->] node [above] {$g_R$} (d0);
		\end{tikzpicture}	
	\end{center}
	Then we define $R_E$, $e_R$, and $m_R$ as the $(\ce,\cm)$-factorization of $g_R$, and $r_E$ as the morphism induced by orthogonality. Following \cite[Proposition~2.1.4]{FK}, $r_E$ is in $\cm$. This endows $E$ with an $\mathbb L$-structure that by construction makes $e$ and $m$ morphisms in $\Str(\LL)$.  
	
	$(2)$. Given a regular epimorphism $g\colon A\to B$ in $\Str(\mathbb L)$ we can consider the factorization $g=me$ as above. But $m$ is a monomorphism and $g$ is in particular an extremal epimorphism; thus $m$ is an isomorphism and $g$ is therefore sent to a map in $\ce$.
\end{proof}

In the proposition below we say that a set of objects $\cp$ of $\cv_\lambda$ is an {\em $\ce$-generator} for $\cv_\lambda$ if for any $X\in\cv_\lambda$ there exists $Y\in\cp$ and a map $e\colon Y\to X$ in $\ce$. 

\begin{lemma}\label{coeq}
	Assume that $\cv_\lambda$ has an $\ce$-generator of $\ce$-projective objects, and let $\mathbb F$ be a functional language with $\ce$-stable input arities. Then every $\lambda$-presentable $\mathbb F$-structure $A$ can be presented as a coequalizer
	$$ 
	\xymatrix@=4pc{
		&  
		FZ'\ar@<0.5ex>[r]^{f}
		\ar@<-0.5ex>[r]_{g}& FZ  \ar[r]^{h} & A
	}
	$$
	of morphisms between free algebras over $\lambda$-presentable objects $Z'$ and $Z$ of $\cv$.
\end{lemma}
\begin{proof}	
Let $\cg$ be the full subcategory of $\Str(\mathbb F)$ spanned by these objects. We need to prove that it consists of all the $\lambda$-presentable objects. Clearly $\cg\subseteq \Str(\mathbb F)_\lambda$ and is a strong generator of $\Str(\mathbb F)$, since it contains the free $\lambda$-presentable objects. Thanks to \cite[Theorem~7.2]{Kel82}, to conclude it is enough to prove that $\cg$ is closed under $\lambda$-small weighted colimits.
	
	Since $\lambda$-small coproducts and copowers of free $\lambda$-presentable objects are still free, it is easy to see that $\cg$ is closed under these colimits. Thus we only need to consider coequalizers. Consider a pair $f,g\colon A\to B$ in $\cg$, then we can consider the solid part of the diagram below
	\begin{center}
		\begin{tikzpicture}[baseline=(current  bounding  box.south), scale=2]

			\node (a) at (0,0) {$FZ$};
			\node (a1) at (1,0) {$A$};
			\node (b) at (0,-0.8) {$FW$};
			\node (b1) at (1,-0.8) {$B$};
			\node (z) at (-1.2,-0.8) {$FW'$};
			
			\path[font=\scriptsize]
			
			(a) edge [->>] node [above] {$r$} (a1)
			(b) edge [->>] node [below] {$q$} (b1)
			
			([yshift=-1.5pt]z.east) edge [->] node [below] {$h$} ([yshift=-1.5pt]b.west)
			([yshift=1.5pt]z.east) edge [->] node [above] {$k$} ([yshift=1.5pt]b.west)
			([xshift=-1.5pt]b.north) edge [dashed, <-] node [left] {$f'$} ([xshift=-1.5pt]a.south)
			([xshift=1.5pt]b.north) edge [dashed, <-] node [right] {$g'$} ([xshift=1.5pt]a.south)
			([xshift=-1.5pt]b1.north) edge [<-] node [left] {$f$} ([xshift=-1.5pt]a1.south)
			([xshift=1.5pt]b1.north) edge [<-] node [right] {$g$} ([xshift=1.5pt]a1.south);
		\end{tikzpicture}
	\end{center}
	where the pair $(h,k)$ presents $B$ as a coequalizer of free $\lambda$-presentable objects, while $Z$ is $\lambda$-presentable $\ce$-projective and $r$ an epimorphism (this exists by hypothesis since $A\in\cg$ and the $\ce$-projectives are an $\ce$-generator in $\cv_\lambda$). Now, by the lemma above, $Uq\in\ce$; therefore using the $\ce$-projectivity of $Z$ we find $f',g'\colon FZ\to FW$ making the square above commute. To conclude it is enough to notice that the coequalizer of the pair
	\begin{center}
		\begin{tikzpicture}[baseline=(current  bounding  box.south), scale=2]

			\node (b) at (0,-0.8) {$FW$};
			\node (z) at (-1.5,-0.8) {$F(W'+Z)$};
			
			\path[font=\scriptsize]
			
			([yshift=-1.5pt]z.east) edge [->] node [below] {$(h,f')$} ([yshift=-1.5pt]b.west)
			([yshift=1.5pt]z.east) edge [->] node [above] {$(k,g')$} ([yshift=1.5pt]b.west);
		\end{tikzpicture}
	\end{center}
	coincides with the coequalizer of $(f,g)$. 
\end{proof}

\begin{propo}\label{pres1}
	Assume that $\cv_\lambda$ has an $\ce$-generator of $\ce$-projective objects. Let $\mathbb L= \mathbb F\cup \mathbb R$ be a $\lambda$-ary language whose function symbols have $\ce$-stable input arities and relation symbols have $\ce$-stable arities. Then every $\lambda$-presentable $\mathbb L$-structure $A$ has a presentation formula in $\LL_{\lambda\lambda}$.
\end{propo}
\begin{proof}
	Given a $\lambda$-presentable $\mathbb L$-structure $A$, we will construct its presentation formula. Since the algebraic reduct $J_{\mathbb F}(A)$ is a $\lambda$-presentable, it is a coequalizer
	$$ 
	\xymatrix@=4pc{
		&  
		FZ'\ar@<0.5ex>[r]^{f}
		\ar@<-0.5ex>[r]_{g}& FZ  \ar[r]^{h} & J_{\mathbb F}(A)
	}
	$$
	of morphisms $f$ and $g$ between free algebras over $\lambda$-presentable objects $Z'$ and $Z$ of $\cv$ (see \ref{coeq}). Since $f$ and $g$ are $(Z,Z')$-ary terms, $A$ satisfies the equation $f=g$. 
	
	Now, by \ref{lambda-pres} there are a $\lambda$-small number of $X_t$-ary relations $R_t$ such that $(R_t)_A$ is an $\ce$-quotient $q_t:Y_t\twoheadrightarrow (R_t)_A$ of some $0\neq Y_t\in \cv_\lambda$. Since by hypothesis $\cv_\lambda$ has an $\ce$-generator of $\ce$-projective objects, we can assume that each $Y_t$ is $\ce$-projective. It follows that for each $t$ we have the following diagram
	\begin{center}
		\begin{tikzpicture}[baseline=(current  bounding  box.south), scale=2]

			\node (d0) at (1,0) {$UFZ^{X_t}$};
			
			\node (a) at (-1,-0.8) {$Y_t$};
			\node (a0) at (0,-0.8) {$(R_t)_A$};
			\node (b0) at (1,-0.8) {$A^{X_t}$};
			
			\path[font=\scriptsize]
			
			(a) edge [->>] node [below] {$q_t$} (a0)
			(a0) edge [>->] node [below] {$(r_t)_A$} (b0)
			(b0) edge [<<-] node [right] {$U(h)^{X_t}$} (d0)
			
			(a) edge [dashed, ->] node [above] {$\tau_t'$} (d0);
		\end{tikzpicture}	
	\end{center}
	in $\cv$, where $\tau_t'$ exists since $U(h)^{X_t}$ is in $\ce$ (by \ref{strong-pres} and since $X_t$ is $\ce$-stable) and $Y_t$ is $\ce$-projective. By transposition $\tau_t'$ corresponds to a map $\tau_t''\colon F(X_t\otimes Y_t)\to FZ$, and hence to a $(Z,X_t\otimes Y_t)$-ary term $\tau_t$. Now for each $t$, consider the $Z$-ary atomic formula $R_t^{Y_t}(\tau_t)$. We shall show that the formula
	$$ \pi^A:=(f=g) \wedge \bigwedge\limits_{t} R_t^{Y_t}(\tau_t) $$
	is a presentation formula for $A$. We need to show that for every $\mathbb L$-structure $B$ we have a bijection, natural in $B$, between elements $a\colon Z\to B$ for which $B\models\pi^A[a]$ and morphisms $\bar a\colon A\to B$ of $\mathbb L$-structures.
	
	Notice first that $(f=g)_B$ is by definition the equalizer of $B^f,B^g\colon B^Z\to B^{Z'}$ and this, since $B^{(-)}\cong \mathbb F\tx{-Str}(F-,J_\mathbb FB)$, is isomorphic to $\mathbb F\tx{-Str}(J_\mathbb FA,J_\mathbb FB)$. Thus we have a natural bijection between elements $a\colon Z\to B$ in $\cv$ for which $B\models(f=g)[a]$ and morphisms of $\mathbb F$-structures $\bar a\colon J_\mathbb FA\to J_\mathbb FB$. To conclude it is enough to show that for such an $a\colon Z\to B$, we have $B\models R_t^{Y_t}(\tau_t)[a]$ for all $t$ if and only if $\bar a$ is actually a morphism of $\mathbb L$-structures.
	
	Given $a\colon Z\to B$, denote by $a'\colon I\to B^Z$ its transpose; by naturality, we know that the transpose of the morphism
	$$ I\xrightarrow{\ a'\ } B^Z\xrightarrow{\ B^\tau} B^{X_t\otimes Y_t} $$
	with respect to $Y_t$, is the composite
	$$ x_t\colon Y_t\xrightarrow{\ q_t\ } (R_t)_A\xrightarrow{\ (r_t)_A\ } A^{X_t}\xrightarrow{\ \bar a^{X_t}} B^{X_t}. $$ 
	Now, since $q_t$ is in $\ce$, the map of $\mathbb F$-structure $\bar a$ extends to a map of $\mathbb L$-structures if and only if $x_t\colon Y_t\to B^{X_t}$ above factors through $(r_t)_B\colon (R_t)_B\rightarrowtail B^{X_t}$ for any $t$ (the morphism $(R_t)_A\to (R_t)_B$ is induced by the orthogonality property). 
	
	On the other hand, the interpretation of $R_t^{Y_t}(\tau_t)$ in $B$ is given by the pullback 
	\begin{center}
		\begin{tikzpicture}[baseline=(current  bounding  box.south), scale=2]
			
			\node (a0) at (0,0.9) {$R_t^{Y_t}(\tau_t)_B$};
			\node (a0') at (0.3,0.7) {$\lrcorner$};
			\node (b0) at (1.4,0.9) {$B^Z$};
			\node (c0) at (0,0) {$(R_t)_B^{Y_t}$};
			\node (d0) at (1.4,0) {$B^{X_t\otimes Y_t}$};
			
			\path[font=\scriptsize]
			
			(a0) edge [>->] node [above] {} (b0)
			(a0) edge [->] node [left] {} (c0)
			(b0) edge [->] node [right] {$B^\tau$} (d0)
			(c0) edge [>->] node [below] {$(r_t)_B^{Y_t}$} (d0);
		\end{tikzpicture}	
	\end{center} 
	where we have identified $B^{X_t\otimes Y_t}$ with $(B^{X_t})^{Y_t}$. Thus $a\colon Z\to B$ is such that $B\models R_t^{Y_t}(\tau_t)[a]$ if and only if $B^\tau\circ a'\colon I\to B^{X_t\otimes Y_t}$ factors through $(r_t)_B^{Y_t}$. By transposing with respect to $Y_t$, that holds if and only if  $x_t\colon Y_t\to B^{X_t}$ factors through $(r_t)_B\colon (R_t)_B\rightarrowtail B^{X_t}$. By the argument above this holds for any $t$ if and only if $\bar a$ extends to a map of $\mathbb L$-structures, concluding the proof.
\end{proof}

\begin{propo}\label{pres2}
	Let $\mathbb L=(\emptyset,\mathbb R)$ be a $\lambda$-ary relational language. Then every $\lambda$-presentable $\mathbb L$-structure $A$ has a presentation formula in $\LL_{\lambda\lambda}$.
\end{propo}
\begin{proof}
	Since the functional part of the language is trivial, by Proposition~\ref{lambda-pres}, an $\LL$-structure $A$ is $\lambda$-presentable if and only if $A$ is $\lambda$-presentable as an object of $\cv$ and there are a $\lambda$-small number of $X_t$-ary relations $R_t$ such that $(R_t)_A$ is an $\ce$-quotient $q_t:Y_t\twoheadrightarrow (R_t)_A$ of some $0\neq Y_t\in \cv_\lambda$.
	
	Thus we a can argue exactly as in Proposition~\ref{pres1} above by taking $f=g=1_A$ and $\tau'_t=(r_t)_A q_t$.
\end{proof}



\begin{exams}\label{examples-presentation}$ $
	{\setlength{\leftmargini}{1.6em}
		\begin{enumerate}
			\item The category $\Met$ with (surjective, isometry) factorization 
			sa\-tis\-fies \ref{assum}. For $\lambda>\omega$, discrete $\lambda$-presentable objects form a $Surj$-generator because $\delta_X:X_0\to X$ are surjective. Since $Surj$ is closed under discrete powers, $e^X$ is a surjection provided that $X$ is discrete and $e$ is a surjection. Hence discrete objects are $Surj$-stable. Therefore \ref{pres1} applies to $\Met$. 
			\item Similarly, \ref{pres1} applies to the category $\Pos$ with the (surjective, embedding) factorization. 
			\item The category $\Ban$ of Banach spaces and linear maps of norm $\leq 1$ with the (strong epimorphisms, monomorphism) factorization sa\-tis\-fies \ref{assum}. Discrete Banach spaces are coproducts of $\mathbb C$. Since $\Ban$ is reflective in the monadic category of totally convex spaces (see \cite{PR}), every Banach space is a regular quotient of a discrete space ($l_1$ is the left adjoint to the forgetful functor from totally convex spaces to sets). Thus discrete $\lambda$-presentable objects for a (strong epimorphism)-generator. 
			Hence \ref{pres1} applies to $\Ban$.  
			\item If $\cv$ is a symmetric monoidal quasivariety as in~\cite{LT20}, then $\cv$ is locally finitely presentable, and we can consider the (regular epi, mono) factorization system. The regular projective objects form a regular generator, hence we can apply~\ref{pres1}.  
		\end{enumerate}
	}
\end{exams}

\begin{rem}\label{depend1}
Assume that we have factorization systems $(\ce_0,\cm_0)$ and $(\ce_1,\cm_1)$ such that $\ce_0\subseteq\ce_1$, so that we have an inclusion 
$$ J \colon \Str(\mathbb L)^1\hookrightarrow\Str(\mathbb L)^0,$$
as in~\ref{depend}. If every $\mathbb L^0$-structure has a presentation formula with respect to $\ce_0$ then every $\LL^1$-structure has it with respect to $\ce_1$ as
well. Indeed, it follows by \ref{depend}, that given $A\in\Str(\LL)^0$, if $\varphi$ is a presentation formula for the $\mathbb L^1$-structure $JA$, then it is also a presentation formula for $A$. However, since $J$ may not preserve the $\lambda$-presentable objects, if $A$ is a $\lambda$-presentable $\LL^0$-structure, the corresponding presentation formula may not be $\lambda$-ary (but just $\mu$-ary, where $\mu$ is such that $JA$ is $\mu$-presentable). Note that there are always big enough $\mu$ for which $J$ preserves $\mu$-presentable objects \cite[Theorem~2.19]{AR}.
\end{rem}

The remark above allows us to consider presentation formulas in $\Met$, $\CMet$ and $\Ban$ for
the factorization system $(\ce_1,\cm_1)$=(dense, isometry). Indeed, one takes $(\ce_0,\cm_0)=(Surj,Inj)$, and considers the presentation formulas obtained in that setting by virtue of~\ref{pres1} and \ref{pres2} (using that discrete objects are $\ce_0$-projective and $\ce_0$-stable); by the remark above, these are also presentation formulas with respect to the (dense, isometry) factorization system.

We conclude this section with the following lemma, which will be useful in Section~\ref{purity-section}. 

\begin{lemma}\label{present-morphism}
	In the setting of ~\ref{pres1} or \ref{pres2}, let $h\colon A\to B$ be a morphism between $\lambda$-presentable $\mathbb L$-structures, then there exist formulas $\pi^A(x)$, $\pi^B(y)$, and $\chi(x)$ in $\LL_{\lambda\lambda}$ such that:\begin{enumerate}
		\item $\pi^A$ and $\pi^B$ are presentation formulas for $A$ and $B$ respectively;
		\item $\chi(x)$ is a positive-primitive formula of the from 
		$$ (\exists y)(\pi^A(x)\ \wedge\ \pi^B(y)\ \wedge\ \tau(y)=\eps(x));$$
		\item every $\mathbb L$-structure $K$ satisfies the sequent
		$$ (\forall x) (\chi(x) \vdash \pi^A(x)); $$
		\item for any $\mathbb L$-structure $K$ the $(\ce,\cm)$ factorization of $\Str(\mathbb L)(h,K)$ is given by 
		$$ \Str(\mathbb L)(B,K)\cong \pi^B_K\xrightarrow{\ \ce\ }  \chi_K\xrightarrow{\ \cm\ } \pi^A_K\cong \Str(\mathbb L)(A,K) $$
		where the map in $\cm$ is induced by (3), and the map in $\ce$ is necessarily unique. 
	\end{enumerate}
	The same holds in the setting of~\ref{depend1} for structures and formulas defined with respect to the factorization system $(\ce_1,\cm_1)$, assuming that $(\ce_0,\cm_0)$ satisfies the hypotheses of ~\ref{pres1} or \ref{pres2} and that $J$ preserves the $\lambda$-presentable objects.
\end{lemma}
\begin{proof}
	Assume first that $\LL$ has non empty functional part (so that we are in the setting of~\ref{pres1}). Let
	$$ 
	\xymatrix@=4pc{
		&  
		FZ_A'\ar@<0.5ex>[r]^{f_A}
		\ar@<-0.5ex>[r]_{g_A}& FZ_A  \ar[r]^{h_A} & J_{\mathbb F}(A)
	}
	$$
	and
	$$ 
	\xymatrix@=4pc{
		&  
		FZB'\ar@<0.5ex>[r]^{fB}
		\ar@<-0.5ex>[r]_{g_B}& FZ_B  \ar[r]^{h_B} & J_{\mathbb F}(B)
	}
	$$
	be the coequalizers from the proof of \ref{pres1}. There is an $\ce$-projective object $X$ and a morphism $e:X\to Z_A$ in $\ce$. Then $FX$
	is projective with respect to $h_B$ (since $Uh_B$ is in $\ce$ by \ref{strong-pres}). Thus, there is $t:FX\to FZ_B$ such that $h_Bt=hh_AF(e)$. Let
	$\tau$ be the $(Z_B,X)$-ary term corresponding to $t$ and $\eps$ be the
	$(Z_A,X)$-ary term corresponding to $Fe$. Then we take $\pi^A$ and $\pi^B$ as in \ref{pres1} and define 
	$$\chi'(x,y):=(\pi^A(x)\ \wedge\ \pi^B(y)\ \wedge\ (\tau(y)=\eps(x))\ ).$$
	and $\chi(x):= (\exists y)\chi'(x,y)$; here, $x$ is a $Z_A$-sorted variable, $y$ is a $Z_B$-sorted variable and the equation $\tau(y)=\eps(x)$ is in the sense of~\ref{equation}. 
	
	When the functional part of $\LL$ is empty (and we are in the setting of~\ref{pres2}) we define $\chi$ in the same way, noting that by \ref{pres2} the chosen coequalizers are trivial now, so we do not need to consider an $\ce$-projective object covering $Z_A=A$ (hence $e=1_{A}$ and $t=h$).
	
	By construction for each $K$ we have $\Str(\mathbb L)(A,K)\cong \pi^A_K$ and $\Str(\mathbb L)(B,K)\cong \pi^B_K$. Consider the solid part of the diagram below, where the squares (I) and (II) are pullbacks by definition, and $\chi''(x,y):= \pi^B(y) \wedge (\tau(y)=\epsilon(x))$.
	\begin{center}
		\begin{tikzpicture}[baseline=(current  bounding  box.south), scale=2]
			
			\node (a0) at (-0.3,2.1) {$\Str(\mathbb L)(B,K)\cong\pi^B_K$};
			\node (a'0) at (2,2.1) {$\pi^B_K$};
			\node (c0) at (-0.3,-0.8) {$\Str(\mathbb L)(A,K)\cong\pi^A_K$};
			
			\node (b0) at (3.6,2.1) {$K^{Z_B}$};
			\node (d0) at (3.6,0.9) {$K^{X}$};
			\node (e0) at (2.5,-0.1) {$K^{Z_A}$};
			
			\node (f0) at (2.5,1.1) {$(\tau=\epsilon)_K$};
			\node (g0) at (1,1.1) {$\chi''_K$};
			\node (g'0) at (-0.3,1.1) {$\chi'_K$};
			
			\node (h0) at (0.8,-0.1) {$\chi_K$};
			
			\node (l0) at (3.1,0.85) {(I)};
			\node (m0) at (2.3,1.6) {(II)};
			
			\path[font=\scriptsize]
			
			(a0) edge [->] node [above] {$\id_{\pi^B_K}$} (a'0)
			(a'0) edge [>->] node [above] {} (b0)
			(a0) edge [bend right=40,->] node [left] {$\pi^h_K:=\Str(\mathbb L)(h,K)$} (c0)
			(b0) edge [->] node [right] {$\tau_K$} (d0)
			(e0) edge [->] node [below] {$\ \ \ \ \ \ \ \quad K^e=\eps_K$} (d0)
			(c0) edge [bend right=20, >->] node [above] {$m$} (e0)
			
			(f0) edge [->] node [above] {} (b0)
			(f0) edge [->] node [above] {} (e0)
			
			(g0) edge [>->] node [above] {} (f0)
			(g0) edge [->] node [left] {} (a'0)	
			
			(g'0) edge [->>] node [right] {$q$} (h0)
			(g'0) edge [>->] node [right] {} (g0)
			(g'0) edge [->] node [right] {} (c0)
			(g'0) edge [->] node [right] {$s$} (a0)
			(h0) edge [>->] node [above] {$n$} (e0)
			
			(c0) edge [dashed, <-<] node [above] {$g$} (h0);
		\end{tikzpicture}	
	\end{center} 
	Note that, commutativity of the outer square, plus the fact that (I) is a pullback, implies that the composite $\pi^h_Km$ factors through $(\tau=\epsilon)_K\to K^{Z_A}$ to give a map $r'\colon \pi^B_K\to (\tau=\epsilon)_K$. Similarly, $r'$ factors through $\chi''_K$ first, and then through $\chi'_K$ providing a morphism $r\colon \pi^B_K\to \chi'_K$ for which $ nqr=m\pi^h_K $ and $sr=\id$. Since $e\in\ce$ is (in particular) an epimorphism, then $K^e$ is a monomorphism, and hence so is $s$ (being obtained by pulling back $K^e$ and composing with a map in $\cm$). It follows that $s$ is an isomorphism with inverse given by $r$. Now, by orthogonality of the factorization system there exists $g\colon \chi_K\to \pi^A_K$ in $\cm$ making the relevant triangles commute. This concludes the proof implying both points (3) and (4).
	
	For the last statement of the lemma in the context of~\ref{depend1}, one considers the formulas $\pi^A(x)$, $\pi^B(y)$, and $\chi(x)$ to be those obtained by applying the result above with respect to the factorization system $(\ce_0,\cm_0)$. By~\ref{depend1} these still satisfy (1) and (2). Then points (3) and (4) can be proved as above.
\end{proof}

\section{Elementary morphisms and purity}\label{purity-section}

In this section we study the connection between the notion of pure morphism introduced in \cite{RTe1} and the purely model theoretic notion of elementary morphism (Definition~\ref{elementary}). 

For this purpose presentation formulas will be essential, therefore we make the assumption below. While this might seem a very restrictive set of conditions, it is worth pointing out that many examples arise just by considering the empty language over $\cv$ or just some relational language (which are always allowed under \ref{assumption} below). We leave it to future work to understand whether these assumptions can be relaxed. 

\begin{assume}\label{assumption} 
	We fix a proper enriched factorization system $(\ce_0,\cm_0)$ on $\cv$ for which $\cm_0$ is closed in $\cv^\to$ under $\lambda$-filtered colimits, and we assume that either of the following two conditions holds:\begin{itemize}
		\item $\LL=\RR$ is a $\lambda$-ary relational language;
		\item $\mathbb L=$ is a $\lambda$-ary language whose function symbols have $\ce_0$-stable input arities and relation symbols have $\ce_0$-stable arities.  Moreover $\cv_\lambda$ has an $\ce_0$-generator of $\ce_0$-projective objects.
	\end{itemize}
For the reminder of this section we fix an enriched factorization system $(\ce,\cm)$ as in~\ref{assum}, such that $\ce_0\subseteq\ce$ and for which the inclusion $ J \colon \Str(\mathbb L)\hookrightarrow\Str(\mathbb L)^0$ preserves $\lambda$-presentable objects. Structures and formulas will be considered with respect to the factorization system $(\ce,\cm)$.
\end{assume}

This assumption holds for any $\cv$ whenever the language is purely relational (one takes $(\ce_0,\cm_0)=(\ce,\cm)$). The second condition is valid whenever $\cv$ is endowed with the $(Surj,Inj)$ factorization system by considering discrete arities (also in this case the two factorization system coincide); or when $\cv$ is a symmetric monoidal quasivariety with the (regular epi, mono) factorization system and $\LL$ has regular-projective arities (again, the two factorization system coincide).
Finally, the second condition is valid in $\Met$ and $\Ban$ with $(\ce,\cm)$=(dense, isometry) for languages whose arities are discrete (here we choose $(\ce_0,\cm_0)=(Surj,Inj)$ for $\Met$ and $(\ce_0,\cm_0)=$(strong epi, mono) for $\Ban$). 

\begin{rem}
	Following Remark~\ref{depend1} (based on Propositions~\ref{pres1} and~\ref{pres2}), every $\lambda$-presentable $\mathbb L$-structure $A$ has a presentation formula in $\LL_{\lambda\lambda}$. This is the main reason why Assumption~\ref{assumption} needed to be made.
\end{rem}

We proceed by introducing a notion that generalizes that of elementary embeddings from the ordinary context. See also \cite[Definition~A.5]{RTe1}.

\begin{defi}\label{elementary}
	Let $f\colon K\to L$ be a morphism in $\Str(\mathbb L)$. We say that $f$ is {\em ele\-men\-tary with respect to a positive-primitive formula $\psi(x)$} if the induced diagram below
	\begin{center}
		\begin{tikzpicture}[baseline=(current  bounding  box.south), scale=2]
			
			\node (a0) at (0,0.8) {$\psi_K$};
			\node (a0') at (0.2,0.6) {$\lrcorner$};
			\node (b0) at (1,0.8) {$\psi_L$};
			\node (c0) at (0,0) {$K^X$};
			\node (d0) at (1,0) {$L^X$};
			
			\path[font=\scriptsize]
			
			(a0) edge [->] node [above] {} (b0)
			(a0) edge [>->] node [left] {} (c0)
			(b0) edge [>->] node [right] {} (d0)
			(c0) edge [->] node [below] {$f^X$} (d0);
		\end{tikzpicture}	
	\end{center}
	is a pullback. The morphism $f$ is called {\em $\lambda$-elementary} if it is elementary with respect to each positive-primitive formula in $\LL_{\lambda\lambda}$.
\end{defi}

The following proposition holds without Assumption~\ref{assumption} (just Assumption~\ref{assum-0} is enough); however~\ref{assumption} will be important for the main result of the section (Proposition~\ref{pure})

\begin{propo}\label{closure-elementary}
	Let $f\colon K\to L$ be $\lambda$-elementary. If $L$ satisfies a sequent of the form
	$$(\forall x)(  \varphi(x) \vdash \psi(x)),$$
	where $\varphi$ and $\psi$ are positive-primitive formulas in $\mathbb L_{\lambda\lambda}$, then $K$ also satisfies the same sequent.
\end{propo}
\begin{proof}
	Consider the solid part of the diagram below
	\begin{center}
		\begin{tikzpicture}[baseline=(current  bounding  box.south), scale=2, on top/.style={preaction={draw=white,-,line width=#1}}, on top/.default=6pt]
			
			\node (a0) at (-0.4,0.75) {$\varphi_K$};
			\node (b0) at (0.6,0.75) {$\varphi_L$};
			
			\node (a'0) at (0.4,1.2) {$\psi_K$};
			\node (b'0) at (1.4,1.2) {$\psi_L$};
			
			\node (c0) at (0,0) {$K^X$};
			\node (d0) at (1,0) {$L^X$};
			
			\path[font=\scriptsize]

			(a'0) edge [->] node [above] {} (b'0)
			(a'0) edge [>->] node [left] {} (c0)
			(b'0) edge [>->] node [right] {} (d0)
			
			(a0) edge [->, on top] node [above] {} (b0)
			(a0) edge [>->] node [left] {} (c0)
			(b0) edge [>->] node [right] {} (d0)
			
			(b0) edge [>->] node [right] {} (b'0)
			(a0) edge [dashed, >->] node [right] {} (a'0)
			
			(c0) edge [->] node [below] {$f^X$} (d0);
		\end{tikzpicture}	
	\end{center}
	where the two vertical squares are pullbacks since $f$ is elementary with respect to $\varphi$ and $\psi$, and the arrow $\varphi_L\to\psi_L$ is induced by the fact that $L$ satisfies the sequent. By the universal property of pullbacks, the dashed arrow above exists, showing that also $K$ satisfies the sequent.
\end{proof}

Ordinarily, the $\lambda$-elementary morphisms of $\LL$-structures can be characterized as those morphisms $g\colon K\to L$ that are {\em $\lambda$-pure} (see \cite{AR}). The notion of purity has recently been extended to the enriched context in \cite{RTe1}, were was studied its connection with enriched injectivity classes. We now recall this notion. 

Given a $\cv$-category $\ck$ and two morphisms $f\colon K\to L$ and $g\colon A\to B$ in it, we denote by $\cp(g,L)$ the $(\ce,\cm)$ factorization below.
\begin{center}
	\begin{tikzpicture}[baseline=(current  bounding  box.south), scale=2]
		
		\node (21) at (0,0) {$\ck(B,L)$};
		\node (22) at (1.4,0.5) {$\cp(g,L)$};
		\node (23) at (2.8,0) {$\ck(A,L)$};
		
		\path[font=\scriptsize]
		
		(21) edge [->>] node [above] {$\ce$} (22)
		(21) edge [->] node [below] {$\ck(g,L)$} (23)
		(22) edge [>->] node [above] {$\cm$} (23);
	\end{tikzpicture}	
\end{center} 
Let then $\cp(g,f)$ be the pullback of $\cp(g,L)$ along $\ck(A,f)$, then there is an induced map $r\colon\ck(B,K)\to\cp(g,f) $ as depicted below.
\begin{center}
	\begin{tikzpicture}[baseline=(current  bounding  box.south), scale=2]
		
		\node (a0) at (-0.5,1.2) {$\ck(B,K)$};
		\node (b0) at (1,1.2) {$\ck(B,L)$};
		\node (c0) at (0,-0.2) {$\ck(A,K)$};
		\node (d0) at (1.5,-0.2) {$\ck(A,L)$};
		\node (b'0) at (0,0.55) {$\cp(g,f)$};
		\node (c'0) at (1.5,0.55) {$\cp(g,L)$};
		
		\path[font=\scriptsize]

		(a0) edge [->] node [above] {$\ck(B,f)$} (b0)
		(b'0) edge [->] node [above] {$r'$} (c'0)
		(b0) edge [->>] node [right] {} (c'0)
		(c'0) edge [>->] node [right] {} (d0)
		(c0) edge [->] node [below] {$\ck(A,f)$} (d0)
		(a0) edge [bend right,->] node [left] {$\ck(g,K)$} (c0)
		(a0) edge [dashed, ->] node [right] {$r$} (b'0)
		(b'0) edge [>->] node [right] {} (c0);
	\end{tikzpicture}	
\end{center}

\begin{defi}[\cite{RTe1}]\label{pure-def1}
	We say that $f\colon K\to L$ is $\ce$-\textit{pure with respect to $g$} if the map $r$ above is in $\ce$. We say that $f$ is $(\lambda,\ce)$-\textit{pure} if it is $\ce$-pure with respect to every $g\colon A\to B$ with $A$ and $B$ $\lambda$-presentable.
\end{defi}

Then, when $\cv=\Set$ with the (epi, mono) factorization system, one recovers the ordinary notion of purity. The cases of $\cv=\Met,\Ban,\omega$-$\CPO$, with the factorization system induced by dense maps, as well as the case where $\cv$ is a symmetric monoidal quasivariety, with the (regular epi, mono) factorization system, were all studied in \cite{RTe1}, where the definition above was unpacked to give a more approachable set of conditions.

The following proposition proves that the correspondence between $\lambda$-elementary and $(\lambda,\ce)$-pure morphisms still holds in the enriched framework.

\begin{propo}\label{pure}
 A morphism $f\colon K\to L$ in $\Str(\mathbb L)$ is $(\ce,\lambda)$-pure if and only if it is $\lambda$-elementary.
\end{propo}
\begin{proof}
	Let $f\colon K\to L$ be elementary with respect to 
positive-primitive formulas, and $g\colon A\to B$ a morphism between $\lambda$-presentable objects. Consider $\pi^A(x)$, $\pi^B(y)$, and $\chi(x)$ as in Lemma~\ref{present-morphism} for $g$. Then the diagram defining the $\ce$-purity of $f$ with respect to $g$ becomes as below
	\begin{center}
		\begin{tikzpicture}[baseline=(current  bounding  box.south), scale=2]
			
			\node (a0) at (-0.5,1.1) {$\pi^B_K$};
			\node (b0) at (0.8,1.1) {$\pi^B_L$};
			\node (c0) at (0,-0.2) {$\pi^A_K$};
			\node (d0) at (1.3,-0.2) {$\pi^A_L$};
			\node (b'0) at (0,0.55) {$\chi_K$};
			\node (a0') at (0.2,0.4) {$\lrcorner$};
			\node (c'0) at (1.3,0.55) {$\chi_L$};
			
			\path[font=\scriptsize]

			(a0) edge [->] node [above] {$\pi^B_f$} (b0)
			(b'0) edge [->] node [above] {$r'$} (c'0)
			(b0) edge [->>] node [right] {} (c'0)
			(c'0) edge [>->] node [right] {} (d0)
			(c0) edge [->] node [below] {$\pi^A_f$} (d0)
			(a0) edge [bend right,->] node [left] {$\pi^g_K$} (c0)
			(a0) edge [dashed, ->] node [right] {$r$} (b'0)
			(b'0) edge [>->] node [right] {} (c0);
		\end{tikzpicture}	
	\end{center}
	where we used Lemma~\ref{present-morphism}(4) for the isomorphism $\cp(g,L)\cong\chi_L$, and that $f$ is elementary with respect to $\chi(x)$ for $\cp(g,f)\cong\chi_K$. The morphisms $\pi^B_f$ and $\pi^B_f$ are those induced in the sense explained at the beginning of Section~\ref{sect:pres}. It follows that $r$ is in $\ce$ again thanks to Lemma~\ref{present-morphism}(4); thus $f$ is $\ce$-pure with respect to $g$.
	
	Conversely, assume that $f\colon K\to L$ is an $(\ce,\lambda)$-pure morphism between $\mathbb L$-structures, and 
	$$\psi(x):= (\exists y) \varphi(x,y) $$
	be any positive-primitive $X$-ary formula, where $\varphi$ is $X+Y$-ary and a conjunction of atomic formulas. 
	By Corollary~\ref{conj->pres} we find a $\lambda$-presentable $\mathbb L$-structure $B$ such that $\varphi(x,y)$ is a presentation formula for $B$. Moreover the $\cv$-natural inclusions
	$$ \Str(\mathbb L)(B,K)\cong \varphi_K\longrightarrow K^{X+Y}\cong \Str(\mathbb L)(FX+FY,K)$$
	are induced from a morphism $h\colon FX+FY\to B$ in $\Str(\mathbb L)$. Then we define $g\colon A\to B$ to be the composite
	$$ A:= FX \xrightarrow{\ i_1\ } FX+FY\xrightarrow{\ h\ } B. $$
	Consider now the diagram involved in the definition of purity. By construction of $B$ and $A$, the definition of $\varphi$ with respect to $\psi$, and since by hypothesis $f$ is $\ce$-pure with respect to $g$, that diagram becomes.
	\begin{center}
		\begin{tikzpicture}[baseline=(current  bounding  box.south), scale=2]
			
			\node (a0) at (-0.5,1.2) {$\varphi_K$};
			\node (b0) at (0.8,1.2) {$\varphi_L$};
			\node (c0) at (0,-0.2) {$K^X$};
			\node (d0) at (1.3,-0.2) {$L^X$};
			\node (b'0) at (0,0.55) {$\cp(g,f)$};
			\node (a0') at (0.2,0.4) {$\lrcorner$};
			\node (c'0) at (1.3,0.55) {$\psi_L$};
			
			\path[font=\scriptsize]

			(a0) edge [->] node [above] {$\varphi_f$} (b0)
			(b'0) edge [->] node [above] {$r'$} (c'0)
			(b0) edge [->>] node [right] {} (c'0)
			(c'0) edge [>->] node [right] {} (d0)
			(c0) edge [->] node [below] {$f^X$} (d0)
			(a0) edge [bend right,->] node [left] {$g^*$} (c0)
			(a0) edge [->>] node [right] {$r$} (b'0)
			(b'0) edge [>->] node [right] {} (c0);
		\end{tikzpicture}	
	\end{center}
	So that $\cp(g,L)$ is the $(\ce,\cm)$ factorization of $g^*$, which by definition is given by the composite
	$$ \varphi_K\to K^{X+Y}\xrightarrow{\pi_1}K^X. $$
	Thus $\cp(g,L)\cong \psi_K$ and hence $f$ is elementary with respect to $\psi$. 
\end{proof}

\section{Regular theories and injectivity}\label{regular-section}

In this section we compare the subcategories of models of enriched regular theories in $\Str(\LL)$ with those subcategories which arise as enriched injectivity classes in the sense of \cite{LR12}. We keep the assumptions of~\ref{assumption} on our base $\cv$ and the enriched factorization system $(\ce,\cm)$ on $\cv$. 

\begin{defi}\label{regular}
	A {\em regular} theory $\mathbb T$ is a set of sequents of the form
	$$(\forall x)(  \varphi(x) \vdash (\exists y)(\psi(x,y)\wedge \varphi(x)))$$
	where $\varphi$ and $\psi$ are conjunctions of atomic formulas in $\mathbb L_{\lambda\lambda}$. We call these sequents \textit{regular}.
\end{defi}

\begin{rem}
	As discussed in the introduction, these is not the most obvious generalization of the notion of regular theory from ordinary logic. This choice was made necessary by the fact that, since $\ce$ may not be stable under pullbacks, certain deduction rules of ordinary regular logic do not hold any more. In particular the sequent above is not in general equivalent to the sequent 
	$$(\forall x)(  \varphi(x) \vdash (\exists y)\psi(x,y))$$
	(such equivalence holds whenever $\ce$ is pullback stable, see Remark~\ref{existential-wedge}). In Theorems~\ref{inj1.5} and~\ref{inj2} we give conditions so that a more natural notion of regular theory can be considered; in~\ref{inj2} these condition still allow $\ce$ not to be pullback-stable.
\end{rem}

Regular theories capture the following sequents:

\begin{exam}
	Any sequent of the form 
	$$(\forall x)(  \varphi(x) \vdash \psi(x)),$$
	with $\varphi$ and $\psi$ conjunctions of atomic formulas, can be seen as a regular sequent.
	
	Indeed, $\psi$ can be though as having arity $X+0$, so that existential quantification on the variable $y:0$ is trivial. This way $(\exists y)(\psi\wedge\varphi)$ is the same as $\psi\wedge\varphi$; and it is easy to see that $A\models (\varphi\vdash \psi\wedge\varphi)$ if and only if $A\models (\varphi\vdash \psi)$ (since no existential quantification is involved). 
\end{exam}

\begin{exam}
	Note that the validity of the formula 
	$$ (\forall x)(\exists y)\ \psi(x,y), $$
	with $\psi$ a conjunction of atomic formulas, is equivalent to that of the regular sequent
	$(\forall x)(  \top \vdash (\exists y)(\psi(x,y)\wedge \top))$
where $\top$ is the empty conjunction.
	Hence this type of sentences can be considered within the framework of regular theories. These include also sentences of the form 
	$$(\forall x)  \varphi(x),$$
	where $\varphi$ is a conjunction of atomic formulas (taking existential quantification over $y: 0$).
\end{exam}

Recall from~\cite{LR12} that, given a morphism $h\colon A\to B$ in a $\cv$-category $\ck$, an object $K$ is $h$-injective (over $\ce$) if the map
$$
\ck(h,K)\colon \ck(B,K)\to\ck(A,K)
$$
lies in $\ce$.  

\begin{defi}
	Given a morphism $h:A\to B$ between $\lambda$-presentable $\mathbb L$-structure $A$ and $B$, we say that a regular sequent $\iota_h$ of $\mathbb L_{\lambda\lambda}$ is an \textit{injectivity sequent} for $h$ if a $\mathbb L$-structure $K$ is $h$-injective if and only if $K\models\iota_h$. 
\end{defi}

Given a set $\ch$ of morphisms between $\lambda$-presentable $\mathbb L$-structures, an $\mathbb L$-structure $K$ is $\ch$-injective if it is 
injective to every $h\in\ch$. Such classes of $\mathbb L$-structures are called $\lambda$-injectivity classes, or $(\lambda,\ce)$-injectivity classes if we want to stress $\ce$.

\begin{theo}\label{inj1}
	Under Assumption~\ref{assumption}, the following are equivalent for a full subcategory $\ca$ of $\Str(\LL)$:\begin{enumerate}
		\item $\ca$ is a $(\lambda,\ce)$-injectivity class in $\Str(\LL)$;
		\item $\ca\cong\Mod(\TT)$ for a regular $\mathbb L_{\lambda\lambda}$-theory $\TT$.
	\end{enumerate}
\end{theo}
\begin{proof}
	$(1)\Rightarrow(2)$. It is enough to show that every morphism between $\lambda$-presentable $\mathbb L$-structures has an injectivity sequent. 
	
	Let $h\colon A\to B$ be a morphism between $\lambda$-presentable $\mathbb L$-structures; then we can consider $\pi^A(x)$, $\pi^B(y)$, and $\chi(x)$ as in Lemma~\ref{present-morphism}. We shall prove that the sequent 
	$$ \iota_h:= (\forall x)(\pi^A(x)\vdash \chi(x) )$$
	is an injectivity sequent for $h$. This sequent is clearly regular. Note that the other sequent is satisfied by any $\mathbb L$-structure $K$ by Lemma~\ref{present-morphism}.
	
	Let $K$ be any $\mathbb L$-structure; then $K$ satisfies $\iota_h$ if and only if  $\chi_K\cong \pi^A_K$ as $\cm$-subobjects of $K^{Z_A}$ (since $K$ always satisfies the other sequent), if and only if the map $\Str(\mathbb L)(h,K)$ is in $\ce$ (by point (4) of Lemma~\ref{present-morphism}), if and only if $K$ is injective with respect to $h$.

	$(2)\Rightarrow(1)$. It is enough to prove that given any sequent
	$$(\forall x)(  \varphi(x) \vdash (\exists y)(\psi(x,y)\wedge \varphi(x))),$$
	where $\varphi$ and $\psi$ are conjunctions of atomic formulas of
	$\mathbb L_{\lambda\lambda}$, there exists a morphism $g\colon A\to B$ between $\lambda$-presentable $\mathbb L$-structures for which: $K$ is $\ce$-injective with respect to $g$ if and only if $K\models \varphi\vdash (\exists y)(\psi\wedge\varphi)$, for any $K\in\Str(\mathbb L)$.
	
	By Corollary~\ref{conj->pres} there exists $\lambda$-presentable $\mathbb L$-structures $A$ and $C$ for which $\varphi(x)$ and $\psi(x,y)$ are presentation formulas for $A$ and $C$ respectively; these come together with maps $e\colon F(X)\to A$ and $e'\colon F(X+Y)\to C$. Consider now the pushout $B$ of $e$ along $e'F(i_X)$, as depicted below.
	\begin{center}
		\begin{tikzpicture}[baseline=(current  bounding  box.south), scale=2]
			
			\node (a0) at (0.1,0.9) {$F(X)$};
			\node (b0) at (1.2,0.9) {$A$};
			\node (c0) at (0.1,0) {$F(X+Y)$};
			\node (d0) at (1.2,0) {$C$};
			\node (1) at (2,0.45) {$B$};
			
			\path[font=\scriptsize]
			
			(a0) edge [->] node [left] {$F(i_X)\ $} (c0)
			(a0) edge [->] node [above] {$e$} (b0)
			(c0) edge [->] node [below] {$e'$} (d0)
			(b0) edge [->] node [above] {$\ g$} (1)
			(d0) edge [->] node [left] {} (1);
		\end{tikzpicture}	
	\end{center} 
	By homming into an $\mathbb L$-structure $K$, using the definition of presentation formula, and taking the $(\ce,\cm)$ factorizations of the arrows corresponding to $e F(i_1)$ and $e' F(j_1)$, we obtain the pullback diagram below.
	\begin{center}
		\begin{tikzpicture}[baseline=(current  bounding  box.south), scale=2]
			
			\node (0) at (-2.5,0.7) {$\Str(\LL)(B,K)\cong (\psi\wedge\varphi)_K$};
			\node (a0) at (0,1.4) {$\Str(\LL)(A,K)\cong \varphi_K$};
			\node (b0) at (2,1.4) {$K^X$};
			\node (c0) at (0.1,0) {$\Str(\LL)(C,K)\cong \psi_K$};
			\node (d0) at (2,0) {$K^{X+Y}$};
			\node (1) at (0.4,0.7) {$(\exists y)(\psi\wedge\varphi)_K$};
			
			\path[font=\scriptsize]
			
			(0) edge [->] node [above] {$\Str(\LL)(g,K)\ \ \ \ \ \ \ \ $} (a0)
			(0) edge [->] node [below] {} (c0)
			(a0) edge [>->] node [above] {} (b0)
			(c0) edge [>->] node [below] {} (d0)
			(d0) edge [->] node [left] {} (b0)
			
			(0) edge [->>] node [left] {} (1)
			(1) edge [>->] node [left] {} (b0);
		\end{tikzpicture}	
	\end{center} 
	Note that by orthogonality we already have an inclusion $(\exists y)(\psi\wedge\varphi)_K\subseteq \varphi_K$.
	
	Now, if $K$ is $\ce$-injective with respect to $g$ then $\Str(\LL)(g,K)$ is in $\ce$; thus, by orthogonality, there exists an arrow $\varphi_K\to(\exists y)(\psi\wedge\varphi)_K$ making the square commute. This shows that $K\models \varphi\vdash (\exists y)(\psi\wedge\varphi)$.
	
	Conversely, if $K$ satisfies the sequent, then $(\exists y)(\psi\wedge\varphi)_K=\varphi_K$ as $\cm$-subobjects of $K^X$ (since we already had the other inclusion). Thus $\Str(\LL)(g,K)$ is in $\ce$.
\end{proof}

As a consequence we can characterize $\ce$-injectivity classes in $\cv$ as classes of models of regular theories:

\begin{coro}
	Let $(\ce,\cm)$ be an enriched proper factorization system on $\cv$, and let $\LL^\emptyset:=\emptyset$ be the empty language. Then, the following are equivalent for a full subcategory $\ca$ of $\cv$:\begin{enumerate}
		\item $\ca$ is a $(\lambda,\ce)$-injectivity class in $\cv$;
		\item $\ca\cong\Mod(\TT)$ for a regular $\mathbb L^\emptyset_{\lambda\lambda}$-theory $\TT$.
	\end{enumerate}
\end{coro}
\begin{proof}
	Assumption~\ref{assumption} is always satisfied for $\LL=\emptyset$, and in that case $\Str(\LL)=\cv$. Thus the result follows from Theorem~\ref{inj1} above.
\end{proof}

In the ordinary case, regular theories are defined as sequents of the form 
$$(\forall x)(  \varphi(x) \vdash \psi(x))$$
where $\varphi$ and $\psi$ are positive-primitive formulas in $\mathbb L_{\lambda\lambda}$. However, we have seen above that these may not classify $\ce$-injectivity classes in $\Str(\LL)$. Below we shall give two conditions on $\ce$ so that such classification will be possible.

The first condition is simply stability of $\ce$ under pullbacks. In this case the proof of the following result is a direct applications of the constructions introduced in this paper:

\begin{theo}\label{inj1.5}
	Let $\ce$ be stable under pullbacks. Under Assumption~\ref{assumption}, the following are equivalent for a full subcategory $\ca$ of $\Str(\LL)$:\begin{enumerate}
		\item $\ca\cong\Mod(\TT)$ for a regular $\mathbb L_{\lambda\lambda}$-theory $\TT$;
		\item $\ca\cong\Mod(\TT)$ for a theory $\TT$ with sequents of the form
		$$(\forall x)(  \varphi(x) \vdash (\exists y)\psi(x,y))$$
		where $\varphi$ and $\psi$ are conjunctions of atomic formulas in $\mathbb L_{\lambda\lambda}$;
		\item $\ca\cong\Mod(\TT)$ for a theory $\TT$ with sequents of the form
		$$(\forall x)(  \varphi(x) \vdash \psi(x))$$
		where $\varphi$ and $\psi$ are positive-primitive formulas in $\mathbb L_{\lambda\lambda}$;
		\item $\ca$ is a $(\lambda,\ce)$-injectivity class in $\Str(\LL)$.
	\end{enumerate}
\end{theo}
\begin{proof}
	The implications $(1)\Rightarrow(2)\Rightarrow(3)$ follow by definition, while $(4)\Rightarrow(1)$ follows from Theorem~\ref{inj1}. 
	
	Thus, we only need to prove $(3)\Rightarrow (4)$. For that it is enough to prove that given any sequent
	$$(\forall x)(  \varphi(x) \vdash \psi(x)),$$
	where $\varphi$ and $\psi$ are positive-primitive formulas of
	$\mathbb L_{\lambda\lambda}$, there exists a morphism $g\colon A\to B$ between $\lambda$-presentable $\mathbb L$-structures for which: $K$ is $\ce$-injective with respect to $g$ if and only if $K\models \varphi\vdash \psi$, for any $K\in\Str(\mathbb L)$.
	
	Let us write $\varphi(x)\equiv (\exists y)\varphi'(x,y)$ and $\psi(x)\equiv (\exists z)\psi'(x,z)$ where $\varphi'$ and $\psi'$ are conjunctions of atomic formulas. By Corollary~\ref{conj->pres} there exists $\lambda$-presentable $\mathbb L$-structures $A$ and $C$ for which $\varphi'$ and $\psi'$ are presentation formulas for $A$ and $C$ respectively; these come together with maps $e\colon F(X+Y)\to A$ and $e'\colon F(X+Z)\to C$. Consider now the pushout $B$ below.
	\begin{center}
		\begin{tikzpicture}[baseline=(current  bounding  box.south), scale=2]
			
			\node (0) at (-1.1,0.45) {$FX$};
			\node (a0) at (0.1,0.9) {$F(X+Y)$};
			\node (b0) at (1.2,0.9) {$A$};
			\node (c0) at (0.1,0) {$F(X+Z)$};
			\node (d0) at (1.2,0) {$C$};
			\node (1) at (2,0.45) {$B$};
			
			\path[font=\scriptsize]
			
			(0) edge [->] node [above] {$F(i_1)\ $} (a0)
			(0) edge [->] node [below] {$F(j_1)\ $} (c0)
			(a0) edge [->] node [above] {$e$} (b0)
			(c0) edge [->] node [below] {$e'$} (d0)
			(b0) edge [->] node [above] {$\ g$} (1)
			(d0) edge [->] node [left] {} (1);
		\end{tikzpicture}	
	\end{center} 
	By homming into an $\mathbb L$-structure $K$, using the definition of presentation formula, and taking the $(\ce,\cm)$ factorizations of the arrows corresponding to $e F(i_1)$ and $e' F(j_1)$, we obtain the pullback diagram below.
	\begin{center}
		\begin{tikzpicture}[baseline=(current  bounding  box.south), scale=2]
			
			\node (0) at (-2,0.45) {$\Str(\LL)(B,K)$};
			\node (a0) at (0.1,0.9) {$\Str(\LL)(A,K)$};
			\node (b0) at (1.5,0.9) {$\varphi_K$};
			\node (c0) at (0.1,0) {$\Str(\LL)(C,K)$};
			\node (d0) at (1.5,0) {$\psi_K$};
			\node (1) at (2.4,0.45) {$K^X$};
			
			\path[font=\scriptsize]
			
			(0) edge [->] node [above] {$\Str(\LL)(g,K)\ \ \ \ \ \ \ \ $} (a0)
			(0) edge [->] node [below] {} (c0)
			(a0) edge [->>] node [above] {} (b0)
			(c0) edge [->>] node [below] {} (d0)
			(b0) edge [>->] node [above] {} (1)
			(d0) edge [>->] node [left] {} (1)
			(b0) edge [dashed, >->] node [left] {$?$} (d0);
		\end{tikzpicture}	
	\end{center} 
	Now, if $K$ is $\ce$-injective with respect to $g$ is follows that the composite $\Str(\LL)(B,K)\to \varphi_K$ above is in $\ce$; thus, by orthogonality, there exists an arrow $\varphi_K\to\psi_K$ making the square commute. This shows that $K\models \varphi\vdash \psi$.
	
	Conversely, if $K$ satisfies the sequent, we have a dashed arrow as in the diagram. Hence $\Str(\LL)(B,K)$ can be seen as the pullback of the cospan 
	$$\Str(\LL)(C,K)\twoheadrightarrow \psi_K\leftarrow \Str(\LL)(A,K)$$ 
	since every map in $\cm$ is a monomorphism. But $\ce$ is pullback stable, thus $\Str(g,K)$ is in $\ce$ and hence $K$ is $\ce$-injective with respect to $g$.
\end{proof}

Pullback stability of $\ce$ is quite restrictive, for instance dense maps in $\Met$ or $\CMet$ do not satisfy it (see \cite[Remark~4.5]{RTe1}).

The second set of conditions we present relies on the notion of purity introduced in \cite{RTe1}, and on the main result of the same paper which classifies injectivity classes in terms of closure under certain constructs. First we need the following result:

\begin{propo}\label{regular1}
	Assume that $\ce$ is closed under products in $\cv^\to$ and 
$\mathbb T$   be a theory consisting of sequents of the form 
	$$(\forall x)(  \varphi(x) \vdash \psi(x))$$
	where $\varphi$ and $\psi$ are positive-primitive formulas in $\mathbb L_{\lambda\lambda}$.
	
	Then $\Mod(\mathbb T)$ is closed in $\Str(\mathbb L)$ under products, powers by $\ce$-stable objects, $\lambda$-directed colimits, and $\lambda$-elementary (equivalently, $(\lambda,\ce)$-pure) subobjects.
\end{propo}
\begin{proof}
	Following Lemma~\ref{product}, $\Mod(\mathbb T)$ is closed under products, powers by $\ce$-stable objects, and $\lambda$-directed colimits. Closure under $\lambda$-elementary subobjects follows from Proposition~\ref{closure-elementary}, and this coincides with closure under $(\lambda,\ce)$-pure ones by Proposition~\ref{pure}.
\end{proof}

\begin{theo}\label{inj2}
	Suppose that Assumption~\ref{assumption} and the assumptions of \cite[Theorem~5.5]{RTe1} hold. The following are equivalent for a full subcategory $\ca$ of $\Str(\LL)$:\begin{enumerate}
		\item $\ca\cong\Mod(\TT)$ for a regular $\mathbb L_{\lambda\lambda}$-theory $\TT$;
		\item $\ca\cong\Mod(\TT)$ for a theory $\TT$ with sequents of the form
		$$(\forall x)(  \varphi(x) \vdash (\exists y)\psi(x,y))$$
		where $\varphi$ and $\psi$ are conjunctions of atomic formulas in $\mathbb L_{\lambda\lambda}$;
		\item $\ca\cong\Mod(\TT)$ for a theory $\TT$ with sequents of the form
		$$(\forall x)(  \varphi(x) \vdash \psi(x))$$
		where $\varphi$ and $\psi$ are positive-primitive formulas in $\mathbb L_{\lambda\lambda}$;
		\item $\ca$ is closed under products, powers by $\ce$-stable objects, $\lambda$-filtered colimits, and $\lambda$-elementary subobjects;
		\item $\ca$ is closed under products, powers by $\ce$-stable objects, $\lambda$-filtered colimits, and $(\lambda,\ce)$-pure subobjects;
		\item $\ca$ is a $(\lambda,\ce)$-injectivity class in $\Str(\LL)$.
	\end{enumerate}
\end{theo}
\begin{proof}
	The implications $(1)\Rightarrow(2)\Rightarrow(3)$ follow by definition, $(3)\Rightarrow(4)$ follows by Proposition~\ref{regular1} above, and $(4)\Rightarrow(5)$ by Proposition~\ref{closure-elementary}. Then, $(5)\Rightarrow(6)$ is given by \cite[Theorem~5.5]{RTe1}, and finally $(6)\Rightarrow(1)$ follows from Theorem~\ref{inj1}.
\end{proof}

This corollary applies both to $\Met$ and $\Ban$ for the factorization system induced by the dense maps; as well as for any symmetric monoidal quasivariety $\cv$ with the (regular epi, mono) factorization system.

Let us explain now how existential quantification is interpreted for one particular base of enrichment.

\begin{rem}
	Consider $\cv=\Met$ and the factorization system given by the dense maps. In this example we spell out what it means for an $\LL$-structure $M$ to satisfy a formula of the form 
	$$ (\forall x)(\exists y)\ \psi(x,y), $$
	with $\psi(x,y): X+Y$ a conjunction of atomic formulas. Since existential quantification is not involved in the definition of $\psi$, we know that $\psi_M\subseteq M^X\times M^Y$ is identified by all the pairs $(a,b)$ for which $M\models\psi[a,b]$. It follows that $((\exists y)\psi)_M$ is given by the closure of the subspace
	$$ \{ a\in M^X \ | \ \exists b\in M^Y\ M\models \psi[a,b] \}\subseteq M^X $$
	under limits of Cauchy sequences.
	As a consequence, $M\models (\forall x)(\exists y)\ \psi(x,y)$ if and only if 
	\begin{center}\textit{
			for each $a\in M^X$ there is a Cauchy sequence $(a_n)_{n\geq 0}\subseteq M^X$\\ and $(b_n)_{n\geq 0}\subseteq M^Y$ (not necessarily Cauchy) such that $a=\lim_na_n$\\ and $M\models \psi[a_n,b_n]$ for each $n$.}
	\end{center}
	This can equivalently be expressed as follows
	\begin{center}\textit{
			for each $a\in M^X$ and for each $\epsilon>0$ there exist $a'\in M^X$ and $b'\in M^Y$\\ such that $d(a,a')<\epsilon$ and $M\models \psi[a',b']$.}
	\end{center}
	The same argument works in the case of $\cv=\CMet$ or $\cv=\Ban$ where the left class of the factorization system is given again by the dense maps.
\end{rem}

\begin{exam}
	Consider again $\cv=\Met$ with the factorization system given by either the surjections or the dense maps. Let $\mathbb I:=[0,\infty]$ be the positive real line with infinity included; we have two maps $i_0,i_\infty\colon 1\to \mathbb I$ which pick out $0$ and $\infty$ respectively. Then to give an element $p\in M^\mathbb I$ is the same as to give a path (of possibly length infinity) between two points in $M$.
	
	We can then consider the following formula 
	$$ (\forall (x,y):1+1) (\exists p:\mathbb I) (p(i_0)(x)=x\ \wedge\ p (i_\infty)(y)=y) $$ 
	on the empty language over $\Met$.
	
	If $\ce$ is the class of surjections, then a metric space $M$ satisfies the sentence above if and only if it is path connected. More interestingly, if $\ce$ is the class of dense maps, a metric space $M$ satisfies the sentence above if and only if 
	\begin{center}\textit{
			for each $a,b\in M$ and for any $\epsilon>0$ there exist $a',b'\in M$ with $d(a,a'),d(b,b')<\epsilon$\\ and a path $p\colon \mathbb I\to M$ for which $p(0)=a'$ and $p(\infty)=b'$.}
	\end{center}
	A space satisfying the condition above need not be path connected; one example is given by the subspace of $\mathbb R^2$ given by
	$$ \{(x,\tx{sin}(1/x))\ |\  x>0 \}\cup (\{0\}\times [-1,1]) $$
	which is one of the canonical examples of a connected space which is not path connected. However, it is easy to see that this space satisfies the condition above.
\end{exam}

\begin{exam}
	Let us take $\cv=\mathbf{DGAb}$, the monoidal category of differentially graded abelian groups, with the (regular epi, mono) factorization system. The unit $I$ of $\mathbf{DGAb}$ is the chain complex with $\mathbb Z$ in degree $0$ and trivial otherwise. Let $P$ be the chain complex with $\mathbb Z$ in degree $1$ and $0$, differential $1_\mathbb Z$ between those, and trivial everywhere else. Then we have an inclusion $i\colon I\to P$. 
	
	Consider now the sentence
	$$ (\forall x) (\exists y)\ (i(y)=x) $$
	in the empty language, where $x:I$ and $y:P$. It is easy to see that a chain complex $A$ satisfies the sentence above if and only if it forms a long exact sequence (that is, if $\tx{Im}(d_A^{n+1})=\tx{Ker}(d_A^n)$ for each $n$).
\end{exam}

\begin{exam}
	Given a small category $\cc$, consider the base of enrichment $\cv=[\cc^{\op},\Set]$ with its cartesian closed structure and the (epi, mono) factorization system.\\ 
	Recall that, if $J$ is a Grothendieck topology on $\cc$, then there is an induced notion of $J$-dense subobject $U\rightarrowtail V$ in $\cv$ (see~\cite[Section~V.1]{MLM}). Consider the regular theory $\TT$ given by the sentences
	$$ (\forall x)\ (\exists! y)\  (my=x),$$
	in the empty language over $\cv$, where $m\colon A\rightarrowtail \cc(-,C)$ is any $J$-dense map with representable codomain. Here the symbol $\exists!$ is interpreted as explained in Notation~\ref{unique}.\\
	Then, using \cite[Theorem~V.4.2]{MLM} and the fact that sheaves are stable under powers, it is easy to see that
	$$ \Mod(\TT)=\tx{Sh}(\cc,J)$$
	is the Grothendieck topos of sheaves over the site $(\cc,J)$.\\
	More generally, one could start with $\cv$ being any Grothendieck topos endowed with the cartesian closed structure and the (epi, mono) factorization system. Then, for any Lawvere-Tierney topology $j$ on $\cv$, we can express the full subcategory $\tx{Sh}_j\cv$ of $\cv$ spanned by the $j$-sheaves as the $\cv$-category of models of a regular theory. 
\end{exam}

\section*{Acknowledgements}

\subsection{Acknowledgements} We thank Joshua Wrigley for valuable feedback.

\subsection{Funding} Both authors acknowledge the support of the Grant Agency of the Czech Republic under the grant 22-02964S. The second author also acknowledges the support of the EPSRC postdoctoral fellowship EP/X027139/1.

\end{document}